\documentclass[pdflatex]{sn-jnl}
\usepackage{mathtools}
\usepackage[numbers]{natbib}
\usepackage{tikz}
\usepackage{hyperref}
\usepackage{graphicx}%
\usepackage{multirow}%
\usepackage{amsmath,amssymb,amsfonts}%
\usepackage{amsthm}%
\usepackage{mathrsfs}%
\usepackage[title]{appendix}%
\usepackage{xcolor}%
\usepackage{textcomp}%
\usepackage{manyfoot}%
\usepackage{booktabs}%
\usepackage{algorithm}%
\usepackage{algorithmicx}%
\usepackage{algpseudocode}%
\usepackage{listings}%
\usepackage{mathrsfs}
\usepackage{xcolor}
\newtheorem{theorem}{Theorem}%
\newtheorem{proposition}{Proposition}%
\newtheorem{example}{Example}%
\newtheorem{remark}{Remark}%
\newtheorem{definition}{Definition}%
\newtheorem{corollary}{Corollary}

\definecolor{darkblue}{RGB}{0,0,139}
\begin{document}

\title[Article Title]{Cramér–Rao Type Bounds for Generalized Exponential Families
}
\begingroup
\renewcommand{\thefootnote}{}
\endgroup

\author*[1]{\fnm{Satyajit} \sur{Dhadumia}}\email{satya22dhadumia@zohomail.in}
\author*[2]{\fnm{M. Ashok Kumar} \sur{}}\email{ ashokm@iitpkd.ac.in}
\affil*[1,2]{\orgdiv{Department of Mathematics}, \orgname{IIT Palakkad},\postcode{ 678623}, \state{Kerala}, \country{India}}
\abstract{
It is well known that the sufficient statistics of an exponential family attain the usual Cramér–Rao Lower Bound (CRLB) for their expected value. In this paper, we extend this result to a broader family of distributions that includes exponential families and power-law families, with Student $t$-distributions as special case. To derive the generalized Cramér–Rao inequality, we first apply Eguchi's theory to Bregman divergence to obtain a Riemannian metric and then follow the approach of Amari and Nagaoka. Since this approach applies only to models with finite support, we use Cauchy–Schwarz inequality to extend it to models with countably infinite or continuous support. The resulting generalized CRLB provides a finite-sample lower bound related to the asymptotic covariance matrix of the Minimum Bregman Divergence Estimator (MBDE). We show that the sufficient statistic of the generalized family attains the generalized CRLB. Moreover, although the generalized CRLB does not, in general, coincide with the asymptotic covariance matrix of the MBDE, the covariance matrix of an estimator attaining the generalized CRLB coincides with the corresponding matrix for the joint distribution of the sample.}


\maketitle

\section{Introduction}
Let $(\mathcal{X},\mu)$ be a measurable space, and let $\mathcal{P}$ denote the set of all probability densities with respect to $\mu$. Throughout, $\mu$ is the counting measure when $\mathcal{X}$ is finite or countably infinite and $\mu$ is the n-dimensional Lebesgue measure when $\mathcal{X}\subseteq \mathbb{R}^n$ is a measurable set with positive Lebesgue measure.

Let $S \subseteq \mathcal{P}$ be a family of probability distributions parameterized by $\theta=[\theta_1,\cdots,\theta_k]^{\top}$, where $k\geq1$, such that
\begin{equation}\label{eqn:statistical model}
    S=\{p_{\theta}:\theta=[\theta_1,\cdots,\theta_k]^{\top}\in \Theta\},
\end{equation}
where $\Theta$ is an open subset of $\mathbb{R}^k$, called the parameter space. Let us also suppose that $\operatorname{Supp}(p_{\theta}) = \mathcal{X}$ for all $\theta\in\Theta$ and the mapping $\theta\mapsto p_{\theta}$ is injective. For example, consider the $n$-dimensional multivariate normal family
\begin{equation}
\label{eqn:multivariate_normal}
p_{\mu,\sigma^2}(\mathrm{x})
=
\frac{1}{(2\pi\sigma^2)^{n/2}}
\exp\left\{
-\frac{1}{2\sigma^2}
\sum_{i=1}^n(x_i-\mu)^2
\right\},\quad\mathrm{x}=[x_1,\ldots,x_n]^{\top}\in\mathbb{R}^n,
\end{equation}
where $\mu\in\mathbb{R}$ and
$\sigma^2>0$. It is a parametric family of probability distributions with parameter space $\Theta = \mathbb{R}\times (0,\infty)$. Also, it is well-known that multivariate normal forms an exponential family.

\begin{definition}[Exponential family]
A family of probability distributions $\mathcal{E} = \{p_{\theta}:\theta\in\Theta\}$, where $\Theta$ is an open subset of $\mathbb{R}^k$, is said to form an exponential family if each $p_{\theta}\in \mathcal{E}$ can be written as 
\begin{equation}
\label{eqn:exponential_family}
p_{\theta}(\mathrm{x}) =
\exp [C(\mathrm{x}) + Z(\theta)+\lambda(\theta)^{\top}T(\mathrm{x})]\qquad
\text{for }\mathrm{x}\in\mathcal{X},
\end{equation}
for some $C:\mathcal{X}\to\mathbb{R}$ and $Z\in C^1(\Theta)$, and $\lambda=[\lambda_1,\ldots,\lambda_k]^{\top}$, $T=[T_1,\ldots,T_k]^{\top}$, where $\lambda_i\in C^1(\Theta)$ and $T_i:\mathcal{X}\to \mathbb{R}$, $i=1,\ldots,k$ \cite[(3b.7.2)]{rao1973linear}.
\end{definition}
For the multivariate normal in \eqref{eqn:multivariate_normal}, 
\begin{align*}
  C(\mathrm{x})=0,~ T(\mathrm{x}) = \left[\frac{1}{n}\sum_{i=1}^n x_i, \frac{1}{n}\sum_{i=1}^n x_i^2\right]^{\top},~ \lambda(\theta) = \left[\frac{n\mu}{\sigma^2}, -\frac{n}{2\sigma^2}\right]^{\top},
\end{align*} and $$Z(\theta) = -\frac{n\mu^2}{2\sigma^2} -\frac{n}{2}\log(2\pi\sigma^2).$$
It is well-known that, if $\mathrm{X}=(X_1,\ldots,X_n)$ is distributed according to $p_\theta\in\mathcal{E}$, by Factorization Theorem \cite[Th.~6.2.6]{casella2024statistical}, $T(\mathrm{X})$ is a sufficient statistic for the parameter $\theta$. Suppose $r(\theta) = \mathbb{E}_{\theta}[T(\mathrm{X})]$. An interesting property of $T(\mathrm{X})$ is that its covariance matrix attains the lower bound in the Cramér-Rao inequality for the parameter $r(\theta)$ \cite[Eq.~(3.13)]{crrao1945}. Cramér-Rao inequality states that, if $\hat{r}(\mathrm{X})$ is an unbiased estimator of $r(\theta)=[r_1(\theta),\ldots,r_k(\theta)]^{\top}$ (under regularity conditions),
\begin{equation}
    \label{eqn:CRLB}
    V_{\theta}[\hat{r}(\mathrm{X})]\succeq J(\theta)[G_n(\theta)]^{-1}J(\theta)^{\top},
\end{equation}
where $G_n(\theta)=\mathbb{E}_{\theta}[u_{\theta}(\mathrm{X})u_{\theta}(\mathrm{X})^{\top}]$ is the Fisher information matrix, $u_{\theta}(x)={\partial}/{\partial{\theta}}\log p_{\theta}(x)$ is the score function  and $J(\theta)$ is the $k \times k$ Jacobian of $r(\theta)$ \cite[Eq.~(3.5)]{crrao1945}, \cite[Th.~7.3.9]{casella2024statistical}\\(Note: Here and throughout, for two matrices $A_1$ and $A_2$, we write $A_1\succeq A_2$ to indicate that the matrix $A_1-A_2$ is positive semi-definite).

Let us now suppose that $\mathrm{X}=(X_1,\ldots,X_n)$ is distributed according to the multivariate Student $t$-distribution $t_{n,\nu}(\mu\mathbf{1}_n,\sigma^2\mathbb{I}_n)$, $\mu\in\mathbb{R}$, $\sigma^2\in (0,\infty)$, with $\nu>4$, where
    \begin{equation}
        \label{eqn:mult_t_dist}
       p_{\mu,\sigma^2}(\mathrm{x})= {\displaystyle {N_{\nu,\sigma^2}}\left[1+\frac{1}{\nu\sigma^2}\sum_{i=1}^n(x_i-\mu)^2\right]^{-(\nu +n)/2}},\quad \mathrm{x}=[x_1,\ldots,x_n]^{\top}\in\mathbb{R}^n,
    \end{equation} where
$$N_{\nu,\sigma^2}=\frac {\Gamma \left[(\nu +n)/2\right]}{\Gamma (\nu /2)\nu ^{n/2} \pi^{n/2}{\sigma}^{n}}.$$ 
The multivariate Student $t$-distribution provides a robust alternative to the multivariate normal distribution for modeling multivariate continuous data with heavy tails or outliers. As $\nu\to\infty$, the multivariate $t$-distribution coincides with the multivariate normal distribution; that is, \eqref{eqn:mult_t_dist} converges to \eqref{eqn:multivariate_normal}. Note that
$$
T(\mathrm{x})=
\left[
\frac{1}{n}\sum_{i=1}^n x_i,\,
\frac{1}{n}\sum_{i=1}^n x_i^2
\right]^{\top}
$$
is a sufficient statistic for $[\mu,\sigma^2]^{\top}$ in the multivariate Student $t$-distribution. \emph{It is therefore now reasonable to ask if this $T(\mathrm{X})$ also satisfies a CRLB}. Indeed, we formulate a more general question. Notice that the Student distribution does not belong to the exponential family but belongs to the more general family of probability distributions called the $\mathbb{B}^{(\alpha)}$-family which we define now.
\begin{definition}[$\mathbb{B}^{(\alpha)}$-family]
Let $\alpha\in\mathbb{R}\setminus\{0\}$. A family of probability distributions $S=\{p_{\theta}:\theta\in\Theta\}$, where $\Theta$ is an open subset of $\mathbb{R}^k$, is said to form a $\mathbb{B}^{(\alpha)}$-family if each $p_{\theta}\in S$ can be represented as
\begin{equation}\label{eqn:alpha_family}
p_{\theta}(\mathrm{x})=\bigl[N(\mathrm{x})+F(\theta)+\omega(\theta)^{\top}T(\mathrm{x})\bigr]^{\frac{1}{\alpha}}\quad \text{for }\mathrm{x}\in\mathcal{X},
\end{equation}for some $N:\mathcal{X}\to\mathbb{R}$ and $F\in C^1(\Theta)$, and $\omega=[\omega_1,\ldots,\omega_k]^{\top}$, $T=[T_1,\ldots,T_k]^{\top}$, where $\omega_i\in C^1(\Theta)$ and $T_i:\mathcal{X}\to \mathbb{R}$, $i=1,\ldots,k$ \cite{gayen2021projection,gayen2023generalized}. 
\end{definition}

 For the multivariate Student $t$-distribution in \eqref{eqn:mult_t_dist},
\begin{align}\label{Student_t_alpha_family}
\alpha&=-\frac{2}{\nu+n},\hspace{2.5cm}
N(\mathrm{x}) = 0, 
 &T(\mathrm{x})= \left[\frac{1}{n}\sum_{i=1}^n x_i,\,
\frac{1}{n}\sum_{i=1}^n x_i^2\right]^\top,\nonumber \\
\omega(\theta) &= \left[
-\frac{2n\mu}{\nu\sigma^2}N_{\nu,\sigma^2}^{-\frac{2}{\nu+n}},\,
\frac{n}{\nu\sigma^2} N_{\nu,\sigma^2}^{-\frac{2}{\nu+n}}
\right]^\top, 
&F(\theta) = N_{\nu,\sigma^2}^{-\frac{2}{\nu+n}}
+ \frac{n\mu^2}{\nu\sigma^2} N_{\nu,\sigma^2}^{-\frac{2}{\nu+n}}.
\end{align}

As $\alpha\to0$, $\mathbb{B}^{(\alpha)}$-family coincides with the exponential family \eqref{eqn:exponential_family} \cite[Remark~5]{gayen2021projection}. In fact, \(\mathbb{B}^{(\alpha)}\)-family is a special case of the more general \(\mathbb{B}^{(\varphi)}\)-family, which we define below.
\begin{definition}[$\mathbb{B}^{(\varphi)}$-family]\label{Def:phi_family}
 Let $\varphi$ be a strictly convex twice differentiable function on $(0,\infty)$, and $\varphi'$ denotes its derivative with respect to its argument. A family of probability distributions $S = \big\{p_{\theta}: \theta\in \Theta\big\}$, where $\Theta$ is an open subset of $\mathbb{R}^k$, is said to form a $\mathbb{B}^{(\varphi)}$-family if each $p_{\theta}$ can be represented as
    \begin{equation}\label{eqn:phi_family}
  p_{\theta}(\mathrm{x})
   =(\varphi')^{-1}\!\left[C(\mathrm{x}) + Z(\theta) + \lambda(\theta)^\top T(\mathrm{x})\right]\quad  \text{for }\mathrm{x}\in\mathcal{X}
    \end{equation} 
   for some functions $C:\mathcal{X}\to\mathbb{R}$ and $Z\in C^1(\Theta)$, and 
$
\lambda=[\lambda_1,\ldots,\lambda_k]^\top,\;
T=[T_1,\ldots,T_k]^{\top},
$ where $\lambda_i\in C^1(\Theta)$ and $T_i:\mathcal{X}\to\mathbb{R}$, $i=1,\ldots,k$.
\end{definition}

\begin{remark}
   Note that $\mathbb{B}^{(\varphi)}$-family is a special case of the $(h,\tau)$-exponential family defined in \cite[Eq.~(4.1)]{MATSUZOE2027131088} with $h(t)=\varphi(t)$, $\tau(t)=t$, $\theta(p)=\lambda(\theta)$, $C(x)=-C(\mathrm{x})$, $\psi(p)=-Z(\theta)$.
\end{remark}
\begin{remark}
Note that, upon choosing 
\begin{equation}
\label{eqn:varphi_function}
\varphi(t)=\varphi_{\alpha}(t)=
\begin{cases}
\frac{t^{1+\alpha}-t+\alpha(1-t)}
{\alpha(1+\alpha)},
&\text{for }\alpha\in\mathbb{R}\setminus\{0,-1\},\\[8pt]
t-\log t-1,
&\text{for }\alpha=-1,
\end{cases}
\end{equation}
with $C(\mathrm{x})=\frac{N(\mathrm{x})-1}{\alpha}$, $Z(\theta)=\frac{F(\theta)}{\alpha}$, $\lambda(\theta)=\frac{\omega(\theta)}{\alpha}$, the $B^{(\varphi)}$-family in \eqref{eqn:phi_family} specializes to the $\mathbb{B}^{(\alpha)}$-family \eqref{eqn:alpha_family}. 
Notice also that the above choice brings \eqref{eqn:alpha_family} to the following form 
\begin{equation}
    p_{\theta}(\mathrm{x})
   =\left[1+\alpha\left(C(\mathrm{x}) + Z(\theta) + \lambda(\theta)^\top T(\mathrm{x})\right)\right]^{\frac{1}{\alpha}}\quad  \text{for }\mathrm{x}\in\mathcal{X}.
\end{equation}
It is now therefore easy to see that \eqref{eqn:alpha_family} coincides with the exponential family as $\alpha\to 0$.
\end{remark}
\begin{remark}
     The exponential family defined in \eqref{eqn:exponential_family} is a $\mathbb{B}^{(\varphi)}$-family with $\varphi(t) = t \log t-t+1$.
\end{remark}
\begin{remark}\label{rem:phi_family_phi_escort_relation}
Note that the $\phi$-exponential family of Jan Naudts \cite{naudts2004estimators} is a special case of the $\mathbb{B}^{(\varphi)}$-family with $\varphi(t)=F_{\phi}(t)-(1-t)F_{\phi}(0)$, $t\in (0,\infty)$, where $F_{\phi}(t)=\int_{1}^t\log_{\phi}(y)dy$ with $ \log_{\phi}(t)\coloneqq\int_{1}^t \frac{1}{\phi(y)}dy$, $C(\mathrm{x})=0$, $Z(\theta)=G(\theta)+F_{\phi}(0)$, $\lambda(\theta)=-\theta^k$, $T(\mathrm{x})=c_k(x)$ (see Appendix A).
\end{remark}
Observe that with $C(\mathrm{x})=C$ for all $\mathrm{x}\in\mathcal{X}$ in $\mathbb{B}^{(\varphi)}$-family \eqref{eqn:phi_family}, by Factorization Theorem, $T (\mathrm{X})$ is a sufficient statistic for $\theta$. 
\emph{It is therefore reasonable to ask whether there exists a finite-sample CRLB, as in the case of the exponential family, that is attained by $T(\mathrm{X})$}. In order to address this question, we return to the motivation of the usual CRLB in the following:

R. A. Fisher established the asymptotic normality of the maximum likelihood estimator (MLE) $\hat{\theta}_{\mathrm{MLE}}$ \cite{fisher1922mathematical}, \cite[Th.~10.1.12]{casella2024statistical}. Specifically, for an i.i.d. sample $X_1,\ldots,X_n$ from $p_\theta$, under suitable regularity conditions,
\begin{equation}
\sqrt{n}\bigl(\hat{\theta}_{\mathrm{MLE}}-\theta\bigr)
\xrightarrow{d}
N\left(0,[G_1(\theta)]^{-1}\right),
\end{equation}
where
\begin{equation}
G_1(\theta)
\coloneqq
\mathbb{E}_{\theta}\left[
u_{\theta}(X)u_{\theta}(X)^{\top}
\right].\nonumber
\end{equation}
Thus, $[G_1(\theta)]^{-1}$ serves as a benchmark for the asymptotic covariance matrix of the MLE. The question of whether such a bound admits a finite-sample lower bound, the Cram\'er-Rao inequality \eqref{eqn:CRLB} was discovered, with the lower bound attained by the sufficient statistics of the exponential family. 

 The asymptotic normality of the minimum Bregman divergence estimator (MBDE), denoted by $\hat{\theta}_{\varphi}$, has been established in \cite{Roy2019DensityPower}, which states that, under certain regularity conditions,
\begin{equation}
\sqrt{n}\bigl(\hat{\theta}_{\varphi}-\theta\bigr)
\xrightarrow{d}
N\left(
0,[I_1^{(\varphi)}(\theta)]^{-1}
\right),
\label{eqn:asym_eff_MBDE}
\end{equation}
where
\begin{equation}\label{eqn:asym_eff_MBDE_FI}
[I_1^{(\varphi)}(\theta)]^{-1}
=[G_1^{(\varphi)}(\theta)]^{-1}
K_1^{(\varphi)}(\theta)
[G_1^{(\varphi)}(\theta)]^{-1},
\end{equation}
with \begin{equation*}
    G_1^{(\varphi)}(\theta)
\coloneqq
\mathbb{E}_{\theta}\left[
\frac{\partial}{\partial\theta}
\varphi'\big(p_{\theta}(X)\big)
u_{\theta}(X)^\top
\right],\quad
K_1^{(\varphi)}(\theta)
\coloneqq
V_{\theta}\left[
p_{\theta}(X)
\varphi''\big(p_{\theta}(X)\big)
u_{\theta}(X)
\right].
\end{equation*}$\hat{\theta}{\varphi}$ is defined as the maximizer of the following likelihood function:
\begin{equation}
    \label{Bregman likelihood function}
    n^{-1}\sum_{i=1}^n \varphi'\big(p_{\theta}(x_i)\big)-\int\Big[p_{\theta}(x)\varphi'\big(p_{\theta}(x)\big)-\varphi\big( p_{\theta}(x)\big)\Big]d\mu(x),
\end{equation}
where $\varphi$ is as in \eqref{eqn:phi_family}. The divergence function\footnote{Divergence is a non-negative function $D$ such that $D(p, q) = 0$ iff $p = q$.} associated with \eqref{Bregman likelihood function} is
\begin{equation}
\label{BregDiv}
B^{(\varphi)}(q,p_{\theta})
\coloneqq
\int_{\mathcal{X}}
\Big\{
\varphi(q(x))
-\varphi(p_{\theta}(x))
-\varphi'(p_{\theta}(x))
\big(q(x)-p_{\theta}(x)\big)
\Big\}
\,d\mu(x).
\end{equation}
With $\varphi(t)=t\log t-t+1$, \eqref{BregDiv} reduces to 
\begin{equation}
    D_{KL}(q,p_{\theta})\coloneqq\int_{\mathcal{X}} q(x)\log\frac{q(x)}{p_{\theta}(x)}d\mu(x),
\end{equation} known as the Kullback--Leibler (KL) divergence, also called the I-divergence or relative entropy. With $\varphi(t)=\varphi_{\alpha}(t)$ as in \eqref{eqn:varphi_function}, $B^{(\varphi)}$ reduces to
\begin{align}
\label{eqn:basu-divergence}
    B^{(\alpha)}(q,p_{\theta})=\int_{\mathcal{X}}\left\{ \frac{{p_{\theta}(x)}^{1+\alpha}}{1+\alpha}-\frac{q(x)p_{\theta}(x)^{\alpha}}{\alpha}+\frac{q(x)^{1+\alpha}}{\alpha (1+\alpha)}\right\}d\mu(x).
\end{align}\eqref{eqn:basu-divergence} is known as BHHJ (Basu--Harris--Hjort--Jones) divergence \cite{basu2011statistical,BasuHarris1998}. The divergence $B^{(\alpha)}$ has appeared under several names in the literature, including power pseudo-distance \cite{broniatowski2012decomposable,MR3013393}, density power divergence \cite{BasuHarris1998,kanamori2014scale,naudts2004estimators}, and $\beta$-divergence \cite{mihoko2002robust}.  The cases $\alpha=0$ and $\alpha=-1$ are obtained by taking the corresponding limits. For $\alpha=0$, the BHHJ divergence reduces to the KL divergence, with the generating function $\varphi(t)=t\log t-t+1$.
For $\alpha=-1$, the BHHJ divergence reduces to the reverse KL divergence:
\begin{equation}
\label{eqn:rev_KL_divergence}
D_{RKL}(q,p_{\theta})
\coloneqq
\int_{\mathcal{X}}
\left\{
\log\frac{p_{\theta}(x)}{q(x)}
+
\frac{q(x)}{p_{\theta}(x)}
-1
\right\}
\,d\mu(x),
\end{equation}
with $\varphi(t)=t-\log t-1$ \cite[Sec.~11.4]{basu2011statistical}. 
Another important special case is $\alpha=1$ for which $\varphi(t)=t^2/2$. The corresponding BHHJ divergence is
\begin{equation}
\label{eqn:L_2_distance}
L_2(q,p_{\theta})
\coloneqq
\frac{1}{2}
\int_{\mathcal{X}}
\left[q(x)-p_{\theta}(x)\right]^2
\,d\mu(x).
\end{equation}
Motivated by the asymptotic covariance matrix expression in \eqref{eqn:asym_eff_MBDE_FI}, and in analogy with the usual CRLB, we propose a generalized Cram\'er-Rao inequality for the covariance matrix of unbiased estimators of $r(\theta)$ with lower bound given by
\begin{equation}\label{eqn:variance_suff_stat_phi_family}
    J(\theta)[I_n^{(\varphi)}(\theta)]^{-1}J(\theta)^T,
\end{equation}
where
 \begin{equation}
[I_n^{(\varphi)}(\theta)]^{-1}
=
[G_n^{(\varphi)}(\theta)]^{-1}
K_n^{(\varphi)}(\theta)
[G_n^{(\varphi)}(\theta)]^{-1},
\end{equation}with
\begin{equation*}
    G_{n}^{(\varphi)}(\theta)
\coloneqq
\mathbb{E}_{\theta}\!\left[
\frac{\partial}{\partial\theta}
\varphi'\!\big(p_{\theta}(\mathrm{X})\big)
\,u_{\theta}(\mathrm{X})^\top
\right], \quad K_{n}^{(\varphi)}(\theta)
\coloneqq
V_{\theta}\!\left[
p_{\theta}(\mathrm{X})
\varphi''\!\big(p_{\theta}(\mathrm{X})\big)
u_{\theta}(\mathrm{X})
\right].
\end{equation*} Furthermore, this lower bound is attained by the sufficient statistic $T(\mathrm{X})$ of the $B^{(\varphi)}$-family.
We call $I_n^{(\varphi)}(\theta)$ the $\varphi$-information matrix.

The usual Cram\'er-Rao inequality \eqref{eqn:CRLB} can be derived as a direct application of Cauchy-Schwarz inequality. Amari and Nagaoka provided an information-geometric derivation of the usual CRLB \cite[Sec.~2.5]{AmariHiroshi2000} for models with finite support. We follow this approach and derive a generalized Cram\'er-Rao inequality whose lower bound ($\varphi$-CRLB) is related to the asymptotic covariance matrix of the MBDE. We then extend it to models with countably infinite or continuous support using Cauchy–Schwarz inequality. Although the $\varphi$-CRLB does not, in general, coincide with the asymptotic covariance matrix of the MBDE evaluated for the joint observation vector $(X_1,\ldots,X_n)$, the covariance matrix of an estimator attaining the $\varphi$-CRLB coincides with the latter. We also prove that the sufficient statistic of the $\mathbb{B}^{(\varphi)}$ family attains the $\varphi$-CRLB and as a consequence has covariance matrix equal to the asymptotic covariance matrix of the MBDE evaluated for $(X_1,\ldots,X_n)$. The rest of the paper is organized as follows. In Section~\ref{sec:phi-CRLB}, we establish the generalized Cram\'er-Rao inequality using an information-geometric approach and derive the $\alpha$-version of it for models with finite support. In Section~\ref{sec:cauchy-schwarz-proof}, we extend the generalized Cramér–Rao inequality to models with countably infinite or continuous support, derive a general equality condition, and show that equality holds if and only if the underlying family is a $B^{(\varphi)}$-family. In Section~\ref{sec:attainment}, we show that an estimator attaining the $\varphi$-CRLB has covariance matrix equal to the asymptotic covariance matrix of the MBDE evaluated for $(X_1,\ldots,X_n)$. In this section, we also show that the sufficient statistics of the $B^{(\varphi)}$-family attain the $\varphi$-CRLB. Finally, in Section~\ref{sec:concluding remarks}, we conclude with a summary and concluding remarks. Proofs of several auxiliary results are deferred to the Appendix.
\section{A Generalized Cram\'er-Rao Inequality from Bregman Divergence using Information Geometry}\label{sec:phi-CRLB}
We start this section by briefly introducing information geometry. In information geometry, a statistical model $S$ with $ k$ unknown parameters is regarded as a $k$-dimensional differentiable manifold, called a statistical manifold. A coordinate system on $S$ is given by a map
\[
\psi:S\to\Theta\subseteq\mathbb{R}^k,
\qquad
\psi(p_\theta)=\theta.
\]
For example, the family of Bernoulli distributions $S=\{\operatorname{Ber}(\theta):\theta \in (0,1)\}$ forms a $1$-dimensional statistical manifold, whereas the family of the multivariate normal in \eqref{eqn:multivariate_normal} forms a $2$-dimensional statistical manifold. The Fisher information matrix at $p_\theta\in S$ is the $k\times k$ matrix
\[
G_n(\theta)=[g_{ij}(\theta)]_{i,j=1}^k,
\]
where \begin{equation}\label{eqn:metric-component}
    g_{ij}(\theta)
\coloneqq
\mathbb{E}_{\theta}
\left[
\partial_i\log p_\theta(\mathrm{X})
\partial_j\log p_\theta(\mathrm{X})
\right],
\end{equation}
and $\partial_i\coloneqq\frac{\partial}{\partial\theta_i}$. Under the usual regularity and identifiability conditions, $G_n(\theta)$ is positive definite and therefore defines an inner product on the tangent space $T_{p_\theta}(S)$. More precisely, for
\[
X=\sum_{i=1}^k x^i\partial_i,
\qquad
Y=\sum_{j=1}^k y^j\partial_j,
\]where \(x^i,y^j\in\mathbb{R}\) for \(i,j=1,\ldots,k\),
in $T_{p_\theta}(S)$, the Fisher--Rao metric is given by
\[
g_{p_\theta}(X,Y)
=
\sum_{i,j=1}^k
g_{ij}(\theta)x^iy^j.
\]
This metric is commonly referred to as the Fisher information metric \cite{crrao1945,AmariHiroshi2000}.

The Fisher information metric can also be obtained directly from the KL divergence. Indeed,
\begin{equation}
\label{eqn:Fisher_from_KL}
g_{ij}(\theta)
=
-\left.\partial_i\partial'_j
D_{KL}(p_\theta,p_{\theta'})
\right|_{\theta'=\theta},
\end{equation}where $\partial'_j \coloneqq \frac{\partial}{\partial\theta'_j}$.
More generally, Eguchi's theory \cite[Sec.~2]{Eguchi1992} (cf.~\cite[Th.~2.6]{eguchi2022minimum}) shows that a sufficiently smooth divergence $D$ on a differentiable manifold induces a Riemannian metric. In local coordinates $\theta=[\theta_1,\ldots,\theta_k]^{\top}$, this metric is given by
\begin{equation}
\label{eqn:Eguchi_metric}
g_{ij}(\theta)
\coloneqq
-\left.
\partial_i\partial'_j
D(p_\theta,p_{\theta'})
\right|_{\theta'=\theta}.
\end{equation}
Based on the Fisher information metric, Amari and Nagaoka \cite[Sec.~2.5]{AmariHiroshi2000} re-derived the usual CRLB for statistical models with finite support. In this section, we extend this information-geometric approach from the KL divergence to general Bregman divergences and derive generalized Cramér--Rao bounds. Related information-geometric and generalized Cramér-Rao bounds have been studied in \cite{kumar2018information,kumar2020cram,AshokMishra2021,mishra2020generalized}.

In this section, we assume that $\mathcal{X}$ is a finite set. Then, the probability simplex $\mathcal{P}$ (strictly positive probability distributions) forms a $|\mathcal{X}|-1$ dimensional statistical manifold. Let $ \theta = [\theta_1, \dots, \theta_k]^{\top}$ be a coordinate system of $\mathcal{P}$. By using Eguchi's theory, we can define a Riemannian metric $g^{(\varphi)}(\theta)$, referred to as the $\varphi$-Fisher information metric on $\mathcal{P}$, derived from the Bregman divergence $B^{(\varphi)}$, characterized by the matrix $G_n^{(\varphi)}(\theta)=[g_{ij}^{(\varphi)}(\theta)]_{i,j=1}^k$ with
\begin{eqnarray}
\label{ij entry of phi fisher matrix}
 g^{(\varphi)}_{ij}(\theta) &\coloneqq& -B^{(\varphi)}(\partial_i|\partial_j)\nonumber\\
    &\coloneqq& -\partial_i\partial'_j B^{(\varphi)}(p_{\theta},p_{\theta'})\Big|_{\theta^{'}=\theta}\nonumber\\
     &=& \partial_i\partial_j'\Big\{\sum_{\mathrm{x}\in\mathcal{X}}\varphi'\big(p_{\theta'}(\mathrm{x})\big)\big(p_{\theta}(\mathrm{x})-p_{\theta'}(\mathrm{x})\big) \Big\}\Big|_{\theta'=\theta}\nonumber\\
     &=&\partial_i\sum_\mathrm{x}\partial_j'\varphi'\big(p_{\theta'}(\mathrm{x})\big)\big(p_{\theta}(\mathrm{x})-p_{\theta'}(\mathrm{x})\big)-\varphi'\big(p_{\theta'}(\mathrm{x})\big)\partial_j' p_{\theta'}(\mathrm{x})\Big|_{\theta'=\theta}\nonumber\\
     &=&\sum_\mathrm{x}\partial_j'\varphi'\big(p_{\theta'}(\mathrm{x})\big)\partial_ip_{\theta}(\mathrm{x})\Big|_{\theta'=\theta}\nonumber\\
     &=&\sum_\mathrm{x}\partial_j\varphi'\big(p_{\theta}(\mathrm{x})\big)\partial_i p_{\theta}(\mathrm{x})\nonumber\\
     &=&\mathbb{E}_{\theta}\left[\partial_i\log p_{\theta}(\mathrm{X})\partial_j\varphi'\big(p_{\theta}(\mathrm{X})\big)\right].
\end{eqnarray}
By following the same setting as in \cite[Sec.~2.5]{AmariHiroshi2000}, we have 
\begin{equation}
    T_{p}^{(e)}(\mathcal{P})\coloneqq\left\{X^{(e)}\mid X\in T_{p}(\mathcal{P})\right\}=\left\{A\in \mathbb{R}^{\mathcal{X}}\mid \mathbb{E}_{p}{[A]}=0\right\},
\end{equation}
where 
$X^{(e)}(\mathrm{x})\coloneqq X^{(m)}(\mathrm{x})/p(\mathrm{x}) \text{ for } \mathrm{x}\in \mathcal{X}$ with $X^{(m)}\in\mathcal{A}_0
\coloneqq\{A\in \mathbb{R}^{\mathcal{X}}\mid \sum_{\mathrm{x}} A(\mathrm{x})=0\}.$
For $p\in \mathcal{P}$, let 
\begin{equation}\label{phi-embedding}
    \gamma(p)=\varphi'(p).
\end{equation} $\gamma$ identifies $\mathcal{P}$ with $\gamma(\mathcal{P})=\{\varphi'(p)\mid p\in \mathcal{P}\}\subset \mathbb{R}^{\mathcal{X}}$.
Under this, we can define 
\begin{equation}
    X^{(\varphi)}(\mathrm{x})=p(\mathrm{x})\varphi''(p(\mathrm{X}))X^{(e)}(\mathrm{x})\quad \text{for }\mathrm{x}\in \mathcal{X},\nonumber
\end{equation}
with $X^{(e)}\in T_p^{(e)}(\mathcal{P})$, and call $X^{(\varphi)}$ the $\varphi$-representation of $X\in T_p(\mathcal{P})$. Let \begin{equation}
    T_p^{(\varphi)}(\mathcal{P})\coloneqq\big\{X^{(\varphi)}\mid X\in T_p(\mathcal{P})\big\}.\nonumber
\end{equation}
The following proposition gives another characterization for $T_p^{(\varphi)}(\mathcal{P})$.
\begin{proposition}
    \begin{equation}
        T_{p}^{(\varphi)}(\mathcal{P}) = \left\{A \in \mathbb{R}^{\mathcal{X}}\mid \mathbb{E}_{p}\bigg[\frac{A}{p(X)\varphi''(p(X))}\bigg] = 0\right\}.  
    \end{equation}
\end{proposition}
\begin{proof}
    Since $X^{(\varphi)}(\mathrm{x})=p(\mathrm{x})\varphi''(p(\mathrm{x}))X^{(e)}(\mathrm{x}),$ $\frac{X^{(\varphi)}(\mathrm{x})}{p(\mathrm{x})\varphi''(p(\mathrm{X}))}=X^{(e)}(\mathrm{x})$. Hence

$$\mathbb{E}_p\bigg[\frac{X^{(\varphi)}}{p(X)\varphi''\big(p(X)\big)}\bigg]=\mathbb{E}_p\big[X^{(e)}\big] = 0.
$$
On the other hand, if $A\in \mathbb{R}^{\mathcal{X}}$ such that $\mathbb{E}_{p}\Big[\frac{A}{p(X)\varphi''(p(X))}\Big] = 0$ then $\frac{A}{p\,\varphi''(p)}\in T_p^{(e)}(\mathcal{P})$. Hence $\frac{A}{p\,\varphi''(p)}=X^{(e)}$ for some $X\in T_{p}(\mathcal{P}).$ This implies $A(\mathrm{x})=p(\mathrm{x})\varphi''(p(\mathrm{x}))X^{(e)}(\mathrm{x})$ and hence $A\in T_p^{(\varphi)}(\mathcal{P})$.
\end{proof}
Now, for any $X,Y\in T_{p}(\mathcal{P})$, we define the $\varphi$-Fisher information metric at a point $p$ on $\mathcal{P}$ as
\begin{align}
\label{eqn:phi_riemannian_metric}
     g^{(\varphi)}(X,Y)\coloneqq\left\langle X,Y\right\rangle_{p}&\coloneq\sum_{\mathrm{x}} X^{(m)}(\mathrm{x})Y^{(\varphi)}(\mathrm{x})=\mathbb{E}_{p}\big[X^{(e)}Y^{(\varphi)}\big].
\end{align}
\begin{remark}
    For BHHJ divergence (that is, $\varphi(t)=\varphi_{\alpha}(t)$ as in \eqref{eqn:varphi_function}), we get the $\alpha$-Fisher information metric $g^{(\alpha)}(\theta)$, characterized by the matrix $G_n^{(\alpha)}(\theta) = [g_{ij}^{(\alpha)}(\theta)]_{i,j=1}^k$, where
\begin{eqnarray}
\label{eqn:alpha-Fisher-information}
    g^{(\alpha)}_{ij}(\theta) = \mathbb{E}_{\theta}\big[p_{\theta}(\mathrm{X})^{\alpha}\partial_i\log p_{\theta}(\mathrm{X}) \partial_j \log p_{\theta}(\mathrm{X})\big]
 \end{eqnarray} 
\cite{dhadumia2025generalized}. In this case, the embedding $\gamma:\mathcal{P}\to \gamma(\mathcal{P})$ in \eqref{phi-embedding} is given by 
\begin{equation}\label{eqn:gamma_embedding}
    \gamma(p) =
\begin{cases} 
     \dfrac{1}{\alpha}\big\{p^{\alpha}-1\big\}&\text{ for } \alpha \in \mathbb{R}\setminus \{0\} \\
      \log p&\text{ for }  \alpha = 0.   
\end{cases}
\end{equation}
Under this embedding, for two tangent vectors $X,Y\in T_p(\mathcal{P})$, one can define the $\alpha$-Fisher information metric at a point $p$ on $\mathcal{P}$ by
\begin{equation}
    g^{(\alpha)}(X,Y)\coloneqq\left\langle X,Y\right\rangle_{p}\coloneq\mathbb{E}_{p}[X^{(e)}Y^{(\alpha)}], 
\end{equation} where $Y^{(\alpha)}\in T_{p}^{(\alpha)}(\mathcal{P})\coloneq\{A\in \mathbb{R}^{\mathcal{X}}\mid \mathbb{E}_{p}[p(\mathrm{X})^{-\alpha}A]=0 \}$.
\end{remark}
Now, for an arbitrary smooth real-valued function $f$ defined on $\mathcal{P}$, the differential $(df)_{p} \in T^*_{p}(\mathcal{P})$ maps a tangent vector $X \in T_{p}(\mathcal{P})$ to $X(f)$. The gradient vector $(\mathrm{grad}\, f)_{p} \in T_{p}(\mathcal{P})$, associated with $(df)_{p}$, satisfies:
\begin{equation}
    \label{eqn:gradient-def}
\left\langle(\mathrm{grad}\, f)_{p}, X \right\rangle_{p} = X(f) = (df)_{p}(X),\quad
(\mathrm{grad}\, f)_{p} = \sum_{i,j} (g^{ij}(p))^{(\varphi)}\, (\partial_jf)_p\, \partial_i, \nonumber
\end{equation}
where $(g^{ij}(p))^{(\varphi)}$ denotes the $(i,j)^{\text{th}}$ entry of the inverse of the matrix $G_n^{(\varphi)}(p)=\left[\langle \partial_i,\partial_j\rangle_p\right]$
and the squared norm of the differential is given by
\begin{equation}
\label{eqn:differential 2-form}
\| (df)_{p} \|_{p}^2 = \| (\mathrm{grad}\, f)_{p} \|_{p}^2 = \sum_{i,j} (g^{ij}(p))^{(\varphi)}\, (\partial_jf)_p\, (\partial_if)_p.
\end{equation}

The following theorem relates the variance of a random variable $A$, computed with respect to the escort distribution
\begin{equation}
    \label{eqn:escort}
p_{\varphi}(\mathrm{x})\coloneqq \big[\varphi''(p(\mathrm{x}))\big]^{-1}\big/\sum_{\mathrm{y}\in \mathcal{X}}\big[\varphi''(p(\mathrm{y}))\big]^{-1}
\end{equation}
to the geometric quantity $\| (d\mathbb{E}[A])_p \|_p^2$.
\begin{theorem}\label{thm:var_sq_norm} 
     For a random variable $A$ (a function from $\mathcal{X}$ to $\mathbb{R}$), let $\mathbb{E}[A]$ denote the mapping $p\mapsto \mathbb{E}_{p}[A]$ on $\mathcal{P}$. Then
\begin{align*}
    V_{\varphi}[A]=\frac{\|(d\mathbb{E}[A])_{p}\|_{p}^{2}}{\sum_{y}\big[\varphi''\big(p(y)\big)\big]^{-1}},
\end{align*}
where the variance $V_{\varphi}[A]\coloneqq V_{p_{\varphi}}[A]\coloneq \mathbb{E}_{\varphi}[(A-\mathbb{E}_{\varphi}[A])^2]$ with $\mathbb{E}_{\varphi}[A]\coloneq \mathbb{E}_{p_{\varphi}}[A]$ and the norm $ \| \cdot \|_{p} $ is induced by the $\varphi$-Fisher information metric.
\end{theorem}
    \begin{proof} See Appendix B.
\end{proof}
\begin{corollary}
  \label{corollarry 1}
 If $S$ is a submanifold of $\mathcal{P}$, then for $p\in S$ $$V_{\varphi}[A]\geq\frac{\|(d\mathbb{E}[A]\mid_{S})_{p}\|_{p}^{2}}{\sum_{y}\big[\varphi''\big(p(y)\big)\big]^{-1}},$$ with equality if and only if
$$A-\mathbb{E}_{\varphi}[A]\in T_{p}^{(\varphi)}(S).$$  
\end{corollary}
\begin{proof} Since, for every smooth function $f$ on $\mathcal{P}$, $(\operatorname{grad}f|_{S})_p$ is the orthogonal projection of $(\operatorname{grad}f)_p$ onto $T_p(S)$, the result follows from Theorem~\ref{thm:var_sq_norm}. 
\end{proof}
Let $r:\Theta\to\mathbb{R}^k$ be a differentiable function given by
\begin{equation*}
    r(\theta)
=
\big[r_1(\theta),\ldots,r_k(\theta)\big]^\top,
\end{equation*}
its Jacobian matrix is defined by
\[
J(\theta)\coloneq \frac{\partial r(\theta)}{\partial \theta}
\coloneq
\big[\partial_j r_i(\theta)\big]_{i,j=1}^k.
\]
The $\varphi$-escort distribution is defined by
\[
p_{\varphi,\theta}(\mathrm{x})
\coloneq
\frac{
\left[\varphi''\big(p_\theta(\mathrm{x})\big)\right]^{-1}
}{
\displaystyle
\sum_{\mathrm{y}\in\mathcal{X}}
\left[\varphi''\big(p_\theta(\mathrm{y})\big)\right]^{-1}
},
\qquad \mathrm{x}\in\mathcal{X}.
\]
For an unbiased estimator
\[
\hat r(\mathrm{X})
=
\big[\hat r_1(\mathrm{X}),\ldots,\hat r_k(\mathrm{X})\big]^\top
\]
of $r(\theta)$, define its covariance matrix with respect to
$p_{\varphi,\theta}$ by
\[
V_{\varphi,\theta}[\hat r(\mathrm{X})]
\coloneq
\left[
\operatorname{Cov}_{\varphi,\theta}
\big(\hat r_i(\mathrm{X}),\hat r_j(\mathrm{X})\big)
\right]_{i,j=1}^k,
\]
where
\[
\operatorname{Cov}_{\varphi,\theta}
\big(\hat r_i(\mathrm{X}),\hat r_j(\mathrm{X})\big)
\coloneq
\mathbb{E}_{\varphi,\theta}
\big[\hat r_i(\mathrm{X})\hat r_j(\mathrm{X})\big]
-
\mathbb{E}_{\varphi,\theta}
\big[\hat r_i(\mathrm{X})\big]\,
\mathbb{E}_{\varphi,\theta}
\big[\hat r_j(\mathrm{X})\big],
\]
with $\mathbb{E}_{\varphi,\theta}$ denoting expectation with respect to
$p_{\varphi,\theta}$.

In the following, we state one of the main results of this paper, a generalized CRLB for a finite sample. The significance of
this result will be discussed in Section~\ref{sec:attainment}.
\begin{theorem}[$\varphi$-CRLB]\label{thm:discrete_phi_crlb}
Let $\mathrm{X}=(X_1,\ldots,X_n)$ be distributed according to $p_\theta$. Let
\[
\hat r(\mathrm{X})
=
\big[\hat r_1(\mathrm{X}),\ldots,\hat r_k(\mathrm{X})\big]^\top
\]
be an unbiased estimator of $r(\theta)$ with respect to $p_\theta$, that is,
\[
\mathbb{E}_{\theta}[\hat r_i(\mathrm{X})]
=
r_i(\theta),
\qquad i=1,\ldots,k.
\]
Suppose that the $\varphi$-Fisher information matrix
$G_n^{(\varphi)}(\theta)\coloneqq\mathbb{E}_{\theta}\left[\frac{\partial}{\partial\theta}\varphi'(p_{\theta}(\mathrm{X}))u_{\theta}(\mathrm{X})^{\top}\right]$ is positive definite for every
$\theta\in\Theta$. Then
\begin{equation}
\label{eqn:discrete_phi-CRLB}
V_{\varphi,\theta}[\hat r(\mathrm{X})]
\succeq
\frac{1}{
\displaystyle
\sum_{\mathrm{y}\in\mathcal{X}}
\left[\varphi''\big(p_\theta(\mathrm{y})\big)\right]^{-1}
}
J(\theta)
\left[G_n^{(\varphi)}(\theta)\right]^{-1}
J(\theta)^\top.
\end{equation}
\end{theorem}

\begin{proof}
Define $A \coloneqq c^\top \hat{r}(\mathrm{X})$, where $\hat{r}(\mathrm{X})$ is an unbiased estimator of $r(\theta)$ with respect to $p_{\theta}$ and $c \in \mathbb{R}^k$. 
    For $p_{\theta}\in S$, we can calculate the following
    \begin{align*}
         \|(d\mathbb{E}[A]\mid_{S})_{p_{\theta}}\|_{p_{\theta}}^2&=\frac{\partial}{\partial \theta}\mathbb{E}_{\theta}[c^{\top}\hat{r}(\mathrm{X})][G_{n}^{(\varphi)}(\theta)]^{-1}\left\{\frac{\partial}{\partial \theta}\mathbb{E}_{\theta}[c^{\top}\hat{r}(\mathrm{X})]\right\}^{\top},\quad \text{by } \eqref{eqn:differential 2-form}\\
         &=c^{\top} \frac{\partial \mathbb{E}_{\theta}[\hat{r}(\mathrm{X})]}{\partial \theta}[G_n^{(\varphi)}(\theta)]^{-1}\left\{\frac{\partial \mathbb{E}_{\theta}[\hat{r}(\mathrm{X})]}{\partial \theta}\right\}^{\top}c\\
         &=c^{\top}J(\theta)[G_n^{(\varphi)}(\theta)]^{-1} J(\theta)^{\top} c,
    \end{align*}
    and
    \begin{equation*}
   V_{\varphi,\theta}[A]=c^{\top}V_{\varphi,\theta}[\hat{r}(\mathrm{X})]c.
    \end{equation*}
    Applying the Corollary \ref{corollarry 1} to the submanifold $S$, we get
 $$c^{\top} V_{\varphi,\theta}[\hat{r}(\mathrm{X})]\;c\geq c^{\top} \frac{1}{\sum_{\mathrm{y}}\big[\varphi''\big(p_{\theta}(\mathrm{y})\big)\big]^{-1}}J(\theta)[G_n^{(\varphi)}(\theta)]^{-1}J(\theta)^\top c.$$
This establishes \eqref{eqn:discrete_phi-CRLB}.
\end{proof}
\begin{remark}$\varphi$-CRLB shares structural similarities with the generalized Cramér–Rao bound introduced by Jan Naudts \cite{naudts2004estimators}, particularly through the appearance of an escort distribution in its formulation. Although Naudts bound is formulated using a general escort distribution $P_{\theta}$, the $\varphi$-escort distribution $p_{\varphi,\theta}$ emerges naturally in the present framework.
\end{remark}
The following proposition provides a necessary and sufficient condition of equality in $\varphi$-CRLB.
\begin{proposition}\label{prop:discrete equality condition}
    Equality holds in \eqref{eqn:discrete_phi-CRLB} for $\hat{r}(\mathrm{X})$ if and only if
\begin{equation}\label{eqn:equality phi CRLB_finite support}
  \hat{r}(\mathrm{x})-\mathbb{E}_{\varphi,\theta}[\hat{r}(\mathrm{X})]=J(\theta)\big[G_n^{(\varphi)}(\theta)\big]^{-1}p_{\theta}(\mathrm{x})\varphi''\big(p_{\theta}(\mathrm{x})\big)u_{\theta}(\mathrm{x}).
\end{equation} 
\end{proposition}
\begin{proof}
See Appendix C.
\end{proof} 
\begin{remark}
   For $\varphi(t)=t \log t-t+1$, the bound~\eqref{eqn:discrete_phi-CRLB} reduces to the usual CRLB \eqref{eqn:CRLB}. Furthermore, the equality condition \eqref{eqn:equality phi CRLB_finite support} coincides with  \begin{equation}\label{eqn:CRLB_equality_condition}
       \hat{r}(\mathrm{x})-\mathbb{E}_{\theta}[\hat{r}(\mathrm{X})]=J(\theta)[G_n(\theta)]^{-1}u_{\theta}(\mathrm{x}),
   \end{equation}which is the equality condition for the usual CRLB.\end{remark}
Theorem~\ref{thm:discrete_phi_crlb} and Proposition~\ref{prop:discrete equality condition} when applied to the function $\varphi(t)=\varphi_{\alpha}(t)$ as in \eqref{eqn:varphi_function} we get the earlier result of ours in the following \cite{dhadumia2025generalized}.
\begin{theorem} [$\alpha$-CRLB] \label{thm:discrete_alpha_crlb}
    Let $\mathrm{X}=(X_1, \dots, X_n)$ be distributed according to $p_{\theta}$. Let $$\hat{r}(\mathrm{X})=[\hat{r}_1(\mathrm{X}),\ldots,\hat{r}_k(\mathrm{X})]^{\top}$$ be an unbiased estimator of $r(\theta)$ with respect to $p_{\theta}$, that is, $$\mathbb{E}_{\theta}[\hat{r}_i(\mathrm{X})]=r_i(\theta),\quad i=1,\ldots,k.$$ Suppose that the $\alpha$-Fisher information matrix $G_n^{(\alpha)}(\theta)=\mathbb{E}_{\theta}[p_{\theta}(\mathrm{X})^{\alpha}u_{\theta}(\mathrm{X})u_{\theta}(\mathrm{X})^{\top}]$, is positive definite for every $\theta\in\Theta$. Then
      \begin{equation}\label{eqn:discrete_alpha-CRLB} 
    V_{\alpha,\theta}[\hat{r}(\mathrm{X})]\succeq \frac{1}{\sum_{\mathrm{y}}p_{\theta}(\mathrm{y})^{1-\alpha}}J(\theta) [G_n^{(\alpha)}(\theta)]^{-1}J(\theta)^{\top},
\end{equation}
where $V_{\alpha,\theta}[\hat{r}(\mathrm{X})]\coloneqq[\operatorname{Cov}_{\alpha,\theta}\big(\hat{r}_i(\mathrm{X}),\hat{r}_j(\mathrm{X}))\big]_{i,j=1}^{k}$ is the covariance matrix and is with respect to the $\alpha$-escort distribution $$p_{\alpha,{\theta}}{(\mathrm{x})}\coloneqq\frac{p_{{\theta}}{(\mathrm{x})}^{1-\alpha}}{\sum_{\mathrm{y}}p_{{\theta}}{(\mathrm{y})}^{1-\alpha}},\quad \mathrm{x}\in\mathcal{X}.$$
Furthermore, equality holds in \eqref{eqn:discrete_alpha-CRLB} if and only if 
  \begin{equation}
      \hat{r}(\mathrm{x})-\mathbb{E}_{\alpha,\theta}[\hat{r}(\mathrm{X})]=J(\theta)\big[G_n^{(\alpha)}(\theta)\big]^{-1}p_{\theta}(\mathrm{x})^{\alpha}u_{\theta}(\mathrm{x}).
\end{equation} 
\end{theorem}
   \begin{remark} 
       For $ \alpha=0$, $G_n^{(\alpha)}(\theta)$ reduces to the usual Fisher information matrix $G_n(\theta)$, and in this case, $\alpha$-CRLB matches the usual CRLB. Moreover, the equality condition reduces to that of the usual CRLB.
 \end{remark}
  We know that the sample mean is an unbiased estimator for the population mean under the Bernoulli model that attains the usual CRLB. The example below shows that the sample mean satisfies $\alpha$-CRLB with strict inequality.
\begin{example}\label{example 1}
   Let \(X_1, \dots, X_n\), $n>2$ be an i.i.d. sample from a Bernoulli distribution with parameter \( \theta \in (0,1) \), where \( p_{\theta}(x) = \theta^x (1-\theta)^{1-x},\,x\in\{0,1\}.\) Let \( \hat{r}(\mathrm{X}) =\overline{X}= \frac{1}{n} \sum_{i=1}^n X_i \), which satisfies $\mathbb{E}_{{\theta}}[\overline{X}]=\theta$. Then
  $$
V_{\alpha,\theta}[\overline{X}] > \frac{1}{\sum_{\mathrm{y}} p_{\theta}(\mathrm{y})^{1-\alpha}} [G_n^{(\alpha)}(\theta)]^{-1}, \quad\alpha \neq 0.\\
$$ 
See Appendix D for proof.
Note that for \( \alpha = 0 \), the above strict inequality becomes equality.
\end{example} 
\section{The General Support Case} \label{sec:cauchy-schwarz-proof}
Although Eguchi's theory applies to statistical models with finite, countably infinite, or continuous supports, Amari-Nagaoka's theory is limited to models with finite support. As a consequence, the bound~\eqref{eqn:discrete_phi-CRLB} applies only to models with finite support. In the following, we provide the conditions under which $\varphi$-CRLB extends to models with countably infinite or continuous support. 

Let $\mathrm{X}=(X_1, \dots, X_n)$ be distributed according to $p_{\theta}$ of support $\mathcal{X}_{\theta}$. Let $\hat{r}(\mathrm{X})=[\hat{r}_1(\mathrm{X}),\ldots,\hat{r}_k(\mathrm{X})]^{\top}$ be an unbiased estimator of the parameter $r(\theta)=[r_1(\theta),\ldots,r_k(\theta)]^{\top}$, a $k$-dimensional differentiable function of $\theta$, with respect to $p_{\theta}$, that is, $\mathbb{E}_{\theta}[\hat{r}_i]=r_i(\theta)$ for $i=1,\ldots,k$. Let $\varphi$ be a function as in \eqref{eqn:phi_family} such that $\int [\varphi''(p_{\theta}(\mathrm{y}))]^{-1} d\mu(\mathrm{y}) < \infty$. We assume the following regularity conditions:
\begin{enumerate}[label=\Roman*.]
\item [I.]$\int p_{\theta}(\mathrm{x})d\mu(\mathrm{x})$ can be differentiated under the integral sign with respect to $\theta$, that is,
  $$
 \frac{\partial}{\partial\theta} \int p_{\theta}(\mathrm{x}) \, d\mu(\mathrm{x})= \int \frac{\partial}{\partial\theta} p_{\theta}(\mathrm{x}) d\mu(\mathrm{x}).
  $$ 

\item [II.] $\int \hat{r}(\mathrm{x})p_{\theta}(\mathrm{x})d\mu(\mathrm{x})$ is finite and can be differentiated under the integral sign with respect to $\theta$, that is, 
   $$\frac{\partial}{\partial \theta}\int \hat{r}(\mathrm{x})p_{\boldsymbol\theta}(\mathrm{x})d\mu(\mathrm{x})=\int \hat{r}(\mathrm{x})\frac{\partial}{\partial\theta} p_{\theta}(\mathrm{x})^{\top} d\mu(\mathrm{x}).$$
\item [III.]
The functions $ \mathbb{X}(\mathrm{x}) \coloneqq \hat{r}(\mathrm{x}) - \mathbb{E}_{\varphi, \theta}[\hat{r}(\mathrm{X})]$ and $\mathbb{Y}(\mathrm{x}) \coloneqq p_{\theta}(\mathrm{x})\varphi''\big(p_{\theta}(\mathrm{x})\big) \, u_{\theta}(\mathrm{x})$ are square-integrable with respect to $p_{\varphi,\theta}$.
\item[IV.]The $\varphi$-Fisher information matrix $G_n^{(\varphi)}(\theta)=[g^{(\varphi)}_{ij}(\theta)]_{i,j=1}^{k}$, where \\$g^{(\varphi)}_{ij}(\theta) \coloneqq\mathbb{E}_{\theta}\Big[\partial_i\log p_{\theta}(\mathrm{X})\partial_j\varphi'(p_{\theta}(\mathrm{X}))\Big]$, is positive definite for every $\theta\in\Theta$.
\end{enumerate}
\begin{theorem}\label{thm:Phi_CRLB proof by cauchy-schwarz}
  Under the assumption of conditions $\text{I}-\text{IV}$ above, we have 
\begin{equation}\label{eqn:phi_crlb}
    V_{\varphi, \theta}[\hat{r}(\mathrm{X})] \succeq
    \frac{1}{\int_{\mathrm{y}} \left[\varphi''\left(p_{\theta}(\mathrm{y})\right)\right]^{-1} d\mu(\mathrm{y})} 
    J(\theta) 
    [G_n^{(\varphi)}(\theta)]^{-1}
    J(\theta)^\top,
\end{equation}where $G_n^{(\varphi)}(\theta)\coloneqq\mathbb{E}_{\theta}\left[\frac{\partial}{\partial\theta}\varphi'(p_{\theta}(\mathrm{X}))u_{\theta}(\mathrm{X})^{\top}\right]$ and $J(\theta)\coloneq\frac{\partial r(\theta)}{\partial\theta}\coloneqq[\partial_j r_i(\theta)]_{i,j=1}^k$. 
Furthermore, equality holds in \eqref{eqn:phi_crlb} if and only if 
    \begin{equation}
    \label{eqn:equality_condition_phi_crlb}
  \hat{r}(\mathrm{x})-\mathbb{E}_{\varphi,\theta}[\hat{r}(\mathrm{X})]=J(\theta)[G_n^{(\varphi)}(\theta)]^{-1}p_{\theta}(\mathrm{x})\varphi''\big(p_{\theta}(\mathrm{x})\big)u_{\theta}(\mathrm{x}).
 \end{equation}
\end{theorem}
\begin{proof}

For any $a \in \mathbb{R}^k$, let $\mathbf{f}(\mathrm{x})=a^{\top}\mathbb{X}(\mathrm{x})$ and $\mathbf{g}(\mathrm{x})=\mathbb{Y}(\mathrm{x})^{\top}b$, where $b=[G_n^{(\varphi)}(\theta)]^{-1}
J(\theta)^{\top} a$. Then $\mathbf{f}$ and $\mathbf{g}$ are square-integrable with respect to $p_{\varphi,\theta}$ by $\text{III}$. Applying Cauchy–Schwarz inequality to the functions $\mathbf{f}$ and $\mathbf{g}$, we get
\begin{equation}
    \label{Cauchy-schwarz ine}
    \big(\mathbb{E}_{\varphi,\theta}\left[\mathbf{f}(\mathrm{X})\cdot \mathbf{g}(\mathrm{X}) \right]\big)^2
\leq 
\mathbb{E}_{\varphi,\theta}\left[\mathbf{f}(\mathrm{X})^2\right] \cdot 
\mathbb{E}_{\varphi,\theta}\left[\mathbf{g}(\mathrm{X}) ^2  \right].
\end{equation}
Observe that
\begin{align*}
   \mathbb{E}_{\varphi,\theta}&[\mathbf{f}(\mathrm{X})\cdot \mathbf{g} (\mathrm{X})]\\=&\int a^{\top}\big(\hat{r}(\mathrm{x})-\mathbb{E}_{\varphi,\theta}[\hat{r}(\mathrm{X})]\big)p_{\theta}(\mathrm{x})\varphi''\big(p_{\theta}(\mathrm{x})\big)u_{\theta}(\mathrm{x})^{\top}b\frac{\big[\varphi''\big(p_{\theta}(\mathrm{x})\big)\big]^{-1}}{\int \big[\varphi''\big(p_{\theta}(\mathrm{y})\big)\big]^{-1}d\mu(\mathrm{y})}d\mu(\mathrm{x})\\
   =& a^{\top}\int\big(\hat{r}(\mathrm{x})-\mathbb{E}_{\varphi,\theta}[\hat{r}(\mathrm{X})]\big)u_{\theta}(\mathrm{x})^{\top}p_{\theta}(\mathrm{x})\frac{1}{\int \big[\varphi''\big(p_{\theta}(\mathrm{y})\big)\big]^{-1}d\mu(\mathrm{y})}d\mu(\mathrm{x})\cdot b\\
    =&a^{\top}\frac{1}{\int \big[\varphi''\big(p_{\theta}(\mathrm{y})\big)\big]^{-1}d\mu(\mathrm{y})}\left[\int\hat{r}(\mathrm{x})\frac{\partial}{\partial\theta}p_{\theta}(\mathrm{x})^{\top}d\mu(\mathrm{x})-\int \mathbb{E}_{\varphi,\theta}[\hat{r}(\mathrm{X})]\frac{\partial}{\partial\theta}p_{\theta}(\mathrm{x})^{\top}d\mu(\mathrm{x})\right]b\\
    \overset{(a)}{=}& a^{\top}\frac{1}{\int \big[\varphi''\big(p_{\theta}(\mathrm{y})\big)\big]^{-1}d\mu(\mathrm{y})}\left\{ \frac{\partial \mathbb{E}_{\theta}[\hat{r}(\mathrm{X})]}{\partial \theta}- \mathbb{E}_{\varphi,\theta}[\hat{r}(\mathrm{X})]\left(\frac{\partial}{\partial\theta}\int  p_{\theta}(\mathrm{x})d\mu(\mathrm{x})\right)^{\top}\right\}b\\
    =&a^{\top}\frac{1}{\int \big[\varphi''\big(p_{\theta}(\mathrm{y})\big)\big]^{-1}d\mu(\mathrm{y})}\frac{\partial \mathbb{E}_{\theta}[\hat{r}(\mathrm{X})]}{\partial \theta} b\\
    =&a^{\top}\frac{1}{\int\big[\varphi''\big(p_{\theta}(\mathrm{y})\big)\big]^{-1}d\mu(\mathrm{y})}J(\theta)\;b,
   \end{align*}
where equality $(a)$ follows from regularity conditions $\text{I}$ and $\text{II}$.
 \begin{equation}
     \mathbb{E}_{\varphi,\theta}\left[\mathbf{f}(\mathrm{X})^2 \right]=a^{\top}\mathbb{E}_{\varphi,\theta}[(\hat{r}(\mathrm{X})-\mathbb{E}_{\varphi,\theta}[\hat{r}(\mathrm{X})])(\hat{r}(\mathrm{X})-\mathbb{E}_{\varphi,\theta}[\hat{r}(\mathrm{X})])^{\top}]a
    =a^{\top}V_{\varphi,\theta}[\hat{r}(\mathrm{X})]a,\nonumber
 \end{equation}which is finite since $\mathbf{f}$ is square-integrable.
\begin{align*}
    \mathbb{E}_{\varphi,\theta}[\mathbf{g}(\mathrm{X})^2]&=b^{\top}\mathbb{E}_{\varphi,\theta}\Big[p_{\theta}(\mathrm{X})\varphi''\big(p_{\theta}(\mathrm{X})\big)u_{\theta}(\mathrm{X})p_{\theta}(\mathrm{X})\varphi''\big(p_{\theta}(\mathrm{X})\big)u_{\theta}(\mathrm{X})^{\top}\Big]b\\
    &=b^{\top}\int \big(p_{\theta}(\mathrm{x})\big)^2\big(\varphi''\big(p_{\theta}(\mathrm{x})
\big)\big)^2u_{\theta}(\mathrm{x})u_{\theta}(\mathrm{x})^{\top}\frac{\big[\varphi''\big(p_{\theta}(\mathrm{x})\big)\big]^{-1}}{\int\big[\varphi''\big(p_{\theta}(\mathrm{y})\big)\big]^{-1}d\mu(\mathrm{y})}d\mu(\mathrm{x})\cdot b\\
   &=b^{\top}\frac{1}{\int\big[\varphi''\big(p_{\theta}(\mathrm{y})\big)\big]^{-1}d\mu(\mathrm{y})}\int \big(p_{\theta}(\mathrm{x})\big)^2\varphi''\big(p_{\theta}(\mathrm{x})\big)u_{\theta}(\mathrm{x})u_{\theta}(\mathrm{x})^{\top}d\mu(\mathrm{x})\cdot b\\
   &=b^{\top}\frac{1}{\int\big[\varphi''\big(p_{\theta}(\mathrm{y})\big)\big]^{-1}d\mu(\mathrm{y})}\int p_{\theta}(\mathrm{x})\frac{\partial}{\partial\theta}\varphi'(p_{\theta}(\mathrm{x}))u_{\theta}(\mathrm{x})^{\top}d\mu(\mathrm{x})\cdot b\\
   &=b^{\top}\frac{1}{\int\big[\varphi''\big(p_{\theta}(\mathrm{y})\big)\big]^{-1}d\mu(\mathrm{y})}G_{n}^{(\varphi)}(\theta)b.
\end{align*}
From \eqref{Cauchy-schwarz ine}, we get
\begin{align}
    \bigg[a^\top  \frac{1}{\int\big[\varphi''\big(p_{\theta}(\mathrm{y})\big)\big]^{-1}d\mu(\mathrm{y})}J(\theta)b\bigg]^2\leq a^{\top} V_{\varphi,\theta}[\hat{r}(\mathrm{X})]a b^{\top}\frac{1}{\int\big[\varphi''\big(p_{\theta}(\mathrm{y})\big)\big]^{-1}d\mu(\mathrm{y})}G_{n}^{(\varphi)}(\theta)b.\nonumber
\end{align}
That is 
\begin{equation}
\label{Cauchy-Swartz in variance terms}
\frac{1}{\int\big[\varphi''\big(p_{\theta}(\mathrm{y})\big)\big]^{-1}d\mu(\mathrm{y})}\bigg[a^\top J(\theta)b\bigg]^2\leq a^{\top} V_{\varphi,\theta}[\hat{r}(\mathrm{X})]ab^{\top}G_{n}^{(\varphi)}(\theta)b.
\end{equation}
Since $b=[G_n^{(\varphi)}(\theta)]^{-1}J(\theta)^{\top}a$, we get
\begin{align}
    \label{Cauchy-Swartz in variance terms LHS}
     \frac{1}{\int\big[\varphi''\big(p_{\theta}(\mathrm{y})\big)\big]^{-1}d\mu(\mathrm{y})}&\bigg[a^\top J(\theta)b\bigg]^2\nonumber\\&
     = \frac{1}{\int\big[\varphi''\big(p_{\theta}(\mathrm{y})\big)\big]^{-1}d\mu(\mathrm{y})}\bigg[a^\top J(\theta)[G_n^{(\varphi)}(\theta)]^{-1}J(\theta)^{\top}a\bigg]^2
\end{align}
and
\begin{align}
\label{Cauchy-Swartz in variance terms RHS}
    a^{\top} V_{\varphi,\theta}[\hat{r}(\mathrm{X})]ab^{\top}G_{n}^{(\varphi)}(\theta)b
    &=a^{\top} V_{\varphi,\theta}[\hat{r}(\mathrm{X})]a a^{\top}J(\theta)[G_n^{(\varphi)}(\theta)]^{-1}G_n^{(\varphi)}(\theta)[G_n^{(\varphi)}(\theta)]^{-1}J(\theta)^{\top}a\nonumber\\
    &=a^{\top} V_{\varphi,\theta}[\hat{r}(\mathrm{X})]a \Big[a^{\top}J(\theta)[G_n^{(\varphi)}(\theta)]^{-1}
   J(\theta)^{\top}a\Big].
\end{align} Since $G_n^{(\varphi)}(\theta)$ is positive definite, so is $[G_n^{(\varphi)}(\theta)]^{-1}$. Therefore $J(\theta)[G_n^{(\varphi)}(\theta)]^{-1}
   J(\theta)^{\top}$ is positive semi-definite.
Hence, using \eqref{Cauchy-Swartz in variance terms LHS} and \eqref{Cauchy-Swartz in variance terms RHS} in \eqref{Cauchy-Swartz in variance terms}, we get 
\begin{equation}
    a^{\top}\frac{1}{\int\big[\varphi''\big(p_{\theta}(\mathrm{y})\big)\big]^{-1}d\mu(\mathrm{y})}J(\theta)[G_n^{(\varphi)}(\theta)]^{-1}
   J(\theta)^{\top}a\leq a^{\top}V_{\varphi,\theta}[\hat{r}(\mathrm{X})]a.
\end{equation}
This establishes $\varphi$-CRLB.\\
Furthermore, by the equality condition of Cauchy–Schwarz inequality \eqref{Cauchy-schwarz ine}, we get equality in \eqref{eqn:phi_crlb} if and only if
\begin{equation}
a^{\top}\left\{\hat{r}(\mathrm{x}) - \mathbb{E}_{\varphi,\theta}[\hat{r}(\mathrm{X})]\right\}
= c(\theta)\, p_{\theta}(\mathrm{x})\, \varphi''\big(p_{\theta}(\mathrm{x})\big)
u_{\theta}(\mathrm{x})^{\top}b, \nonumber
\end{equation}
where $c(\theta)$ is a constant independent of $\mathrm{x}$.
That is, 
\begin{align}
a^{\top}\left\{\hat{r}(\mathrm{x}) - \mathbb{E}_{\varphi,\theta}[\hat{r}(\mathrm{X})]\right\}
&=  c(\theta)\, p_{\theta}(\mathrm{x})\, \varphi''\big(p_{\theta}(\mathrm{x})\big)
u_{\theta}(\mathrm{x})^{\top}[G_n^{(\varphi)}(\theta)]^{-1}J(\theta)^{\top}a\nonumber\\
&=c(\theta)\frac{\partial}{\partial \theta} \varphi'(p_{\theta}(\mathrm{x}))^{\top}[G_n^{(\varphi)}(\theta)]^{-1}J(\theta)^{\top}a\nonumber\\
&=a^{\top}c(\theta)J(\theta)[G_n^{(\varphi)}(\theta)]^{-1}\frac{\partial}{\partial \theta} \varphi'(p_{\theta}(\mathrm{x})),\label{eqn:eq_cond_intermed}
\end{align}
where \eqref{eqn:eq_cond_intermed} is obtained by using the fact that  $s^\top = s$ for scalar $s$. Since $a$ is arbitrary, \eqref{eqn:eq_cond_intermed} gives 
\begin{equation}\label{cond:Cauchy_swartz_eq}
 \hat{r}(\mathrm{x}) - \mathbb{E}_{\varphi,\theta}[\hat{r}(\mathrm{X})]  =c(\theta)J(\theta)[G_n^{(\varphi)}(\theta)]^{-1}\frac{\partial}{\partial \theta} \varphi'(p_{\theta}(\mathrm{x})).
\end{equation}
Hence
\begin{equation*}
    \hat{r}(\mathrm{x})u_{\theta}(\mathrm{x})^{\top} - \mathbb{E}_{\varphi,\theta}[\hat{r}(\mathrm{X})]  u_{\theta}(\mathrm{x})^{\top}=c(\theta)J(\theta)[G_n^{(\varphi)}(\theta)]^{-1}\frac{\partial}{\partial \theta} \varphi'(p_{\theta}(\mathrm{x}))u_{\theta}(\mathrm{x})^{\top}.
\end{equation*}
Taking expectation with respect to $p_{\theta}$, we obtain
\begin{align*}
    \int   \hat{r}(\mathrm{x})u_{\theta}(\mathrm{x})^{\top}p_{\theta}(\mathrm{x})d\mu(\mathrm{x})-&\mathbb{E}_{\varphi,\theta}[\hat{r}(\mathrm{X})]  \int u_{\theta}(\mathrm{x})^{\top}p_{\theta}(\mathrm{x})d\mu(\mathrm{x})\\
    =&c(\theta)J(\theta)[G_n^{(\varphi)}(\theta)]^{-1}\mathbb{E}_{\theta}\left[\frac{\partial}{\partial \theta} \varphi'(p_{\theta}(\mathrm{x}))u_{\theta}(\mathrm{X})^{\top}\right]\\
    =&c(\theta)J(\theta)[G_n^{(\varphi)}(\theta)]^{-1}G_n^{(\varphi)}(\theta).
\end{align*}This implies
\begin{equation*}
    \frac{\partial \mathbb{E}_{\theta}[\hat{r}(\mathrm{X})]}{\partial \theta}=c(\theta)J(\theta).
\end{equation*}
Hence \begin{equation*}
 J(\theta)=c(\theta)J(\theta).
\end{equation*}
Substituting this in \eqref{cond:Cauchy_swartz_eq}, we obtain
\begin{equation*}
    \hat{r}(\mathrm{x}) - \mathbb{E}_{\varphi,\theta}[\hat{r}(\mathrm{X})]  =J(\theta)[G_n^{(\varphi)}(\theta)]^{-1}\frac{\partial}{\partial \theta} \varphi'(p_{\theta}(\mathrm{x})),
\end{equation*} which is the equality condition \eqref{eqn:equality_condition_phi_crlb}. This completes the proof.
\end{proof}
The regularity conditions $\text{III-IV}$, for the function $\varphi(t)=\varphi_{\alpha}(t)$ as in \eqref{eqn:varphi_function} are the following:
\begin{enumerate}[label=\Roman*.]
    \item [A.] The functions $\hat{r}(\mathrm{x}) - \mathbb{E}_{\alpha, \theta}[\hat{r}(\mathrm{X})]$ and $p_{\theta}(\mathrm{x})^{\alpha} \, u_{\theta}(\mathrm{x})$ are square-integrable with respect to $p_{\alpha,\theta}$. 
  \item[B.] The $\alpha$-Fisher information matrix $G_n^{(\alpha)}(\theta)=[g^{(\alpha)}_{ij}(\theta)]_{i,j=1}^k$, where\\ $g^{(\alpha)}_{ij}(\theta)\coloneqq\mathbb{E}_{\theta}\big[p_{\theta}(\mathrm{X})^{\alpha}\partial_i\log p_{\theta}(\mathrm{X}) \partial_j \log p_{\theta}(\mathrm{X})\big]$, is positive definite for every $\theta\in\Theta$.
\end{enumerate}
In the following theorem, we establish the continuous version of the bound \eqref{eqn:discrete_alpha-CRLB}. 
\begin{theorem}\label{thm:alpha_CRLB_ext}
 Let $\mathrm{X}=(X_1, \ldots, X_n)$ be distributed according to $p_{\theta}$. Suppose we are interested in estimating $r(\theta)=[r_1(\theta),\ldots,r_k(\theta)]^{\top}$, a $k$-dimensional differentiable function of $\theta$.
Let $\hat{r}(\mathrm{X})=[\hat{r}_1(\mathrm{X}),\ldots,\hat{r}_k(\mathrm{X})]^{\top}$ be an unbiased estimator of $r(\theta)$ with respect to $p_{\theta}$, that is, $\mathbb{E}_{\theta}[\hat{r}_i]=r_i(\theta)$ for $i=1,\ldots,k$. Under the regularity conditions $\text{I}-\text{II}$ and $\text{A}-\text{B}$, we have
\begin{equation}
\label{eqn:alpha_CRLB_ext}
    V_{\alpha,\theta}[\hat{r}(\mathrm{X})]\succeq \frac{1}{\int p_{\theta}(\mathrm{y})^{1-\alpha}d\mu(\mathrm{y})}J(\theta) [G_n^{(\alpha)}(\theta)]^{-1}J(\theta)^{\top}.
\end{equation}
Furthermore, equality holds in \eqref{eqn:alpha_CRLB_ext} if and only if 
 \begin{equation}
    \label{eqn:alpha-eq-cond}
 \hat{r}(\mathrm{x})-\mathbb{E}_{\alpha,\theta}[\hat{r}(\mathrm{X})]=J(\theta)[G_n^{(\alpha)}(\theta)]^{-1}p_{\theta}(\mathrm{x})^{\alpha}u_{\theta}(\mathrm{x}).
  \end{equation}
\end{theorem}
\begin{proof}
    Follows from Theorem~\ref{thm:Phi_CRLB proof by cauchy-schwarz} with $\varphi(t)=\varphi_{\alpha}(t)$ as in \eqref{eqn:varphi_function}.
\end{proof}
The following example illustrates a situation in which the $\alpha$-Fisher information is finite while the usual Fisher information diverges. Consequently, the usual CRLB is not applicable. 

\begin{example}\label{exm:equality}
  Let $X_1,\ldots,X_n$ be an i.i.d. sample from 
    \begin{equation}\label{eqn:triangular_pdf}
p_{\theta}(x)=
\begin{cases}
\dfrac{2}{\theta}\left(1-\dfrac{x}{\theta}\right), & 0\le x<\theta,\\[6pt]
0, & \text{otherwise}.
\end{cases}
\end{equation} Choose $\hat{r}(\mathrm{X})=3\overline{X}$. Observe that $\mathbb{E}_{\theta}[3\overline{X}]=\theta$.\\ \\
For $\alpha<2$,
$$V_{\alpha,\theta}[3\overline{X}]=\frac{9(2-\alpha)\theta^2}{n(4-\alpha)(3-\alpha)^2}$$ and for $\alpha>0$,
$$G_n^{(\alpha)}(\theta)=\frac{n2^{(\alpha+1)n}(n\alpha^3+\alpha+2)}{\alpha(\alpha+1)^2(\alpha+2)^n\theta^{\alpha n+2}}.$$
Note that as $\alpha\to 0$, $G_n^{(\alpha)}(\theta)$ diverges. So, the usual CRLB is not useful in this case.\\ Hence, for $0<\alpha<2$, we have 
L.H.S. of $\alpha$-CRLB
\begin{equation*}
    V_{\alpha,\theta}[3\overline{X}]=\frac{9(2-\alpha)\theta^2}{n(4-\alpha)(3-\alpha)^2}
\end{equation*} and 
R.H.S. of $\alpha$-CRLB
\begin{equation*}
     \frac{1}{\int_0^{\theta}\cdots\int_0^{\theta}p_{\theta}(\mathrm{x})^{1-\alpha}d\mathrm{x}}[G_n^{(\alpha)}(\theta)]^{-1}= \frac{(2-\alpha)^n\alpha(\alpha+1)^2(\alpha+2)\theta^2}{2^{2n}n(n\alpha^3+\alpha+2)}.
\end{equation*}
Now observe that, for $n=1$ and $\alpha=1$, equality is attained in the $\alpha$-CRLB. In this case, the lower bound is $3\theta^2/4$, which is attained by $\hat{r}(\mathrm{X})=3X$.\\
For $\alpha=1$ and $n=1$, we have
\[
\mathbb{E}_{\alpha,\theta}[\hat{r}(\mathrm{X})]=\frac{3\theta}{2},
\qquad
G_1^{(\alpha)}(\theta)=\frac{4}{3\theta^3},
\qquad
p_{\theta}(x)^{\alpha}u_{\theta}(x)=\frac{4x}{\theta^3}-\frac{2}{\theta^2}.
\]
Note that
\[
\hat{r}(\mathrm{x})-\mathbb{E}_{\alpha,\theta}[\hat{r}(\mathrm{X})]
=
3x-\frac{3\theta}{2}
=
\frac{3\theta^3}{4}
\left(
\frac{4x}{\theta^3}-\frac{2}{\theta^2}
\right)
=
\bigl[G_1^{(\alpha)}(\theta)\bigr]^{-1}
p_{\theta}(x)^{\alpha}u_{\theta}(x),
\]
which is the equality condition \eqref{eqn:alpha-eq-cond}. For details of calculations, see Appendix E.
\end{example}
\begin{remark}
 The above example is valid only when $0<\alpha<2$. For \(\alpha\leq 0\), \(G_n^{(\alpha)}(\theta)\) diverges. For \(\alpha\geq 2\), \(\int_0^{\theta}\cdots\int_0^{\theta}p_{\theta}(\mathrm{x})^{1-\alpha}d\mathrm{x}\) is not finite. See Appendix E.
\end{remark}
\begin{remark}
   Note that the density function $p_{\theta}$ in the above example is the same as the one considered in Example 1 of \cite{naudts2004estimators}. For this example, the escort density ensuring the optimality of Naudts bound for the estimator $\hat{r}(\mathrm{X})=3X$, denoted by $P_{\theta}$ \cite[Example 1 continued, p.~9]{naudts2004estimators}, coincides with $p_{\alpha,\theta}$ at $\alpha=1$.
\end{remark}
It is well known that equality holds in the usual CRLB if and only if the underlying family of distributions forms an exponential
family \cite{fend1959attainment, wijsman1973attainment, muller1989attainment}. In the following theorem, we establish an analogous characterization of equality in the $\varphi$-CRLB in terms of the $\mathbb{B}^{(\varphi)}$-family, under the following additional assumptions:
\begin{enumerate}[label=\Roman*.]
      \item[V.]  The support of $p_{\theta}$ is independent of $\theta$. That is, $\mathcal{X}_{\theta}=\mathcal{X}$ for all $\theta\in\Theta$. 
     \item[VI.] For each $\mathrm{x}\in \mathcal{X}$, $\varphi'(p_{\theta}(\mathrm{x}))$ is twice continuously differentiable in $\mathbb{\theta}$.
    \item[VII.] $V_{\varphi,\theta}[\hat{r}(\mathrm{X})]$ is positive definite for every $\theta\in\Theta$.
\end{enumerate} Although the following theorem holds for the vector parameter case, we provide the proof for the scalar case ($k=1$) for the sake of convenience.
\begin{theorem}\label{prop:phi equality family} Let $\mathrm{X}=(X_1, \ldots, X_n)$ be distributed according to $p_{\theta}$. Let $\hat{r}(\mathrm{X})=[\hat{r}_1(\mathrm{X}),\ldots,\hat{r}_k(\mathrm{X})]^{\top}$ be an unbiased estimator of $r(\theta)=[r_1(\theta),\ldots,r_k(\theta)]^{\top}$, a $k$-dimensional differentiable function of $\theta$, with respect to $p_{\theta}$. Suppose conditions $\text{I}-\text{VII}$ are satisfied. Then $\hat{r}(\mathrm{X})$ attains the $\varphi$-CRLB if and only if
\begin{equation}\label{eqn:vector_phi_phi family_equality}
      p_{\theta}( \mathrm{x})=(\varphi')^{-1}\left[C(\mathrm{x})+Z(\theta)+\lambda(\theta)^\top \hat{r}(\mathrm{x})\right]\quad\text{for }\mathrm{x}\in \mathcal{X},
  \end{equation} for some functions $C:\mathcal{X}\to \mathbb{R}$ and $Z\in C^1(\Theta)$, and $\lambda=[\lambda_1,\ldots,\lambda_k]^{\top}$, where $\lambda_i\in C^1(\Theta)$, $i=1,\ldots,k$, such that the $k\times k$ Jacobian matrix $R(\theta)\coloneqq\frac{\partial\lambda(\theta)}{\partial\theta}\coloneqq[\partial_j\lambda_i(\theta)]_{i,j=1}^{k}$ is invertible.
\end{theorem}
\begin{proof}
Assume that equality holds in \eqref{eqn:phi_crlb}. Then, from \eqref{eqn:equality_condition_phi_crlb}, we have
\begin{equation}\label{eq:001}
    \hat{r}(\mathrm{x}) = a(\theta)\frac{\partial}{\partial\theta}\varphi'(p_{\theta}(\mathrm{x}))+b(\theta),
\end{equation}
with \begin{equation*}
    a(\theta) =
 r'(\theta)[G_n^{(\varphi)}(\theta)]^{-1} \quad\text{and}\quad b(\theta) = \mathbb{E}_{\varphi,\theta}[\hat{r}(\mathrm{X})].
\end{equation*}
Since $V_{\varphi,\theta}[\hat{r}(\mathrm{X})]>0$, it follows that
$\hat{r}(\mathrm{x})$ is not constant. Hence, $a(\theta)\neq 0$.
Therefore, from \eqref{eq:001}, we obtain
\begin{equation}
    \frac{\partial}{\partial\theta}
    \varphi'(p_{\theta}(\mathrm{x}))
    =a(\theta)^{-1}\hat{r}(\mathrm{x})
    -a(\theta)^{-1}b(\theta).
\end{equation}
Integrating the above equation with respect to $\theta$ yields
\eqref{eqn:vector_phi_phi family_equality}.

Conversely, let us assume that \eqref{eqn:vector_phi_phi family_equality} holds. That is,
\begin{equation*}
    \varphi'(p_{\theta}(\mathrm{x}))
= C(\mathrm{x}) + Z(\theta) + \lambda(\theta) \hat{r}(\mathrm{x}),\quad \mathrm{x}\in \mathcal{X}.
\end{equation*} Then
\begin{equation}\label{eqn:derivative_of_phi_function}
    \frac{\partial}{\partial \theta}\varphi'(p_{\theta}(\mathrm{x}))=Z'(\theta)+  \lambda'(\theta)\hat{r}(\mathrm{x}).
\end{equation}
That is, \begin{equation*}
    \varphi''(p_{\theta}(\mathrm{x}))\frac{\partial}{\partial \theta}p_{\theta}(\mathrm{x})=Z'(\theta)+  \lambda'(\theta)\hat{r}(\mathrm{x}).
\end{equation*}
Hence
\begin{align}
\label{eqn:p-2}
 u_{\theta}(\mathrm{x})=\frac{1}{p_{\theta}(\mathrm{x})\varphi''\big(p_{\theta}(\mathrm{x})\big)}  Z'(\theta)+  \lambda'(\theta)\frac{1}{p_{\theta}(\mathrm{x})\varphi''\big(p_{\theta}(\mathrm{x})\big)}\hat{r}(\mathrm{x}).
\end{align}
Since $\mathbb{E}_{\theta}[u_{\theta}(\mathrm{X})]=0$, it follows that
\begin{align*}
Z'(\theta)\int\big[\varphi''\big(p_{\theta}(\mathrm{x})\big)\big]^{-1}d\mu(\mathrm{x})+ \lambda'(\theta)\mathbb{E}_{\theta}\left[\frac{\hat{r}(\mathrm{X})}{p_{\theta}(\mathrm{X})\varphi''\big(p_{\theta}(\mathrm{X})\big)}\right]&=0.
\end{align*}This implies that
\begin{equation}
Z'(\theta)\int \big[\varphi''\big(p_{\theta}(\mathrm{x})\big)\big]^{-1}d\mu(\mathrm{x})+ \lambda'(\theta)\mathbb{E}_{\varphi, \theta}[\hat{r}(\mathrm{X})]\int \big[\varphi''\big(p_{\theta}(\mathrm{x})\big)\big]^{-1}d\mu(\mathrm{x})=0.\nonumber
\end{equation}That is, 
\begin{equation}\label{eqn:j0}
Z'(\theta)=-  \lambda'(\theta)\mathbb{E}_{\varphi, \theta}[\hat{r}(\mathrm{X})].
\end{equation}
Thus \eqref{eqn:derivative_of_phi_function} becomes
\begin{equation}
   \frac{\partial}{\partial \theta}\varphi'(p_{\theta}(\mathrm{x}))=   \lambda'(\theta)\left(\hat{r}(\mathrm{x})-\mathbb{E}_{\varphi,\theta}[\hat{r}(\mathrm{X})]\right).
\end{equation}
Since $\lambda'(\theta)\neq0$, we get 
\begin{equation}\label{eqn:p2}
  \left[\lambda'(\theta)\right]^{-1}   \frac{\partial}{\partial \theta}\varphi'(p_{\theta}(\mathrm{x}))= \hat{r}(\mathrm{x})-\mathbb{E}_{\varphi,\theta}[\hat{r}(\mathrm{X})].
\end{equation}
Observe that
\begin{align}
    G_{n}^{(\varphi)}(\theta)&=\mathbb{E}_{\theta}\left[\frac{\partial}{\partial\theta}\varphi'(p_{\theta}(\mathrm{X}))u_{\theta}(\mathrm{X})\right]\nonumber\\
    &=\int \frac{\partial}{\partial\theta}\varphi'(p_{\theta}(\mathrm{x}))u_{\theta}(\mathrm{x}) p_{\theta}(\mathrm{x})d\mu(\mathrm{x})\nonumber\\
   &\overset{(a)}{=}\int \left\{Z'(\theta)+  \lambda'(\theta)\hat{r}(\mathrm{x})\right\}\frac{\partial}{\partial\theta}  p_{\theta}(\mathrm{x})d\mu(\mathrm{x})\nonumber\\
    &= Z'(\theta)\int\frac{\partial}{\partial\theta}  p_{\theta}(\mathrm{x}) d\mu(\mathrm{x})+ \lambda'(\theta)\int \hat{r}(\mathrm{x})\frac{\partial}{\partial\theta}  p_{\theta}(\mathrm{x}) d\mu(\mathrm{x})\nonumber\\
    &=   \lambda'(\theta)\frac{\partial \mathbb{E}_{\theta}[\hat{r}(\mathrm{X})]}{\partial\theta}
    = \lambda'(\theta)r'(\theta),\label{eqn:Phi_fisher_inf_&_expon_family}
\end{align}
where the equality $(a)$ follows from \eqref{eqn:derivative_of_phi_function}. This implies 
\begin{equation*}
     \left[\lambda'(\theta)\right]^{-1}=r'(\theta)[G_n^{(\varphi)}(\theta)]^{-1}.
\end{equation*}
Thus, from \eqref{eqn:p2}, we get 
\begin{equation*}
   r'(\theta)[G_n^{(\varphi)}(\theta)]^{-1}\frac{\partial}{\partial \theta}\varphi'(p_{\theta}(\mathrm{x}))=\hat{r}(\mathrm{x})-\mathbb{E}_{\varphi,\theta}[\hat{r}(\mathrm{X})],
\end{equation*}
the condition for equality in $\varphi$-CRLB.
This completes the proof.
\end{proof}
\begin{remark}
  For $\varphi(t)=t\log t-t+1$, \eqref{eqn:vector_phi_phi family_equality} coincides with the exponential family
  which is the equality condition for the usual CRLB \cite{fend1959attainment, wijsman1973attainment, muller1989attainment}. 
\end{remark}
Due to the $\varphi$-family representation of \eqref{eqn:vector_phi_phi family_equality}, the quantities involved in the $\varphi$-CRLB can be computed using derivatives rather than direct integration. We now consider the canonical form of \eqref{eqn:vector_phi_phi family_equality} and observe that it provides a simpler framework for these calculations, and facilitates the evaluation of the $\varphi$-CRLB.
\begin{proposition}\label{prop:canonical_phi_family}
    The canonical form of \eqref{eqn:vector_phi_phi family_equality} is as follows
\begin{equation}\label{canonical_phi_family}
    p_{\lambda}(\mathrm{x})=(\varphi')^{-1}\![C(\mathrm{x})+H(\lambda)+\lambda^{\top}\hat{r}(\mathrm{x})] \quad \text{for }\mathrm{x}\in\mathcal{X},
\end{equation}where $\lambda=[\lambda_1,\ldots,\lambda_k]^{\top}$ are the canonical parameters. Under the regularity condition $I$ for $p_{\lambda}$, we get the following:
\begin{itemize}
\item[(i)]\begin{align*}
\mathbb{E}_{\varphi,\lambda}[\hat{r}(\mathrm{X})]
&=
-\frac{\partial H(\lambda)}{\partial \lambda},
\qquad
G_n^{(\varphi)}(\lambda)
=
\frac{\partial \mathbb{E}_{\lambda}[\hat{r}(\mathrm{X})]}{\partial \lambda},\\
V_{\varphi,\lambda}[\hat{r}(\mathrm{X})]
&=
\frac{1}{\int [\varphi''(p_{\lambda}(\mathrm{y}))]^{-1}d\mu(\mathrm{y})}\frac{\partial \mathbb{E}_{\lambda}[\hat{r}(\mathrm{X})]}{\partial \lambda}.
\end{align*}
\item[(ii)] $ V_{\varphi,\lambda}[\hat{r}(\mathrm{X})]=\frac{1}{\int [\varphi''(p_{\theta}(\mathrm{y}))]^{-1}d\mu(\mathrm{y})}J(\theta)[G_n^{(\varphi)}(\theta)]^{-1}J(\theta)^{\top}.$
\end{itemize}
Observe that for $\varphi(t)=t\log t-t+1,$ \eqref{canonical_phi_family} reduces to the canonical exponential family form \cite[Ch.~1, Sec.~4]{Lehman1983}, \cite{rao1973linear}, where 
\begin{equation*}
    \mathbb{E}_{\lambda}[\hat{r}(\mathrm{X})]=-\frac{\partial H(\lambda)}{\partial \lambda},\quad G_n(\lambda)=V_{\lambda}[\hat{r}(\mathrm{X})]=-\frac{\partial^2}{\partial\lambda\partial \lambda^{\top}}H(\lambda)
\end{equation*}and 
\begin{equation*}
    V_{\lambda}[\hat{r}(\mathrm{X})]=J(\theta)
[G_n(\theta)]^{-1}J(\theta).
\end{equation*}For details, see Appendix F.
\end{proposition}


 \begin{remark}\label{remark:eq_alpha_crlb_family}
Theorem~\ref{prop:phi equality family} for $\varphi(t)=\varphi_{\alpha}(t)$ as in \eqref{eqn:varphi_function} gives the equality condition for $\alpha$-CRLB. That is, for $\alpha\in\mathbb{R}\setminus\{0\}$, $p_{\theta}$ forms a $\mathbb{B}^{(\alpha)}$-family as follows:
 \begin{equation}\label{eqn:join_alpha_family}
      p_{\theta}(\mathrm{x})=\left[N(\mathrm{x})+F(\theta)+\omega(\theta)^{\top}\hat{r}(\mathrm{x})\right]^{1/\alpha}\quad\text{for  }\mathrm{x}\in \mathcal{X}
   \end{equation}
where $N(\mathrm{x})=1+\alpha C(\mathrm{x})$, $F(\theta)=\alpha Z(\theta)$, $\omega(\theta)=\alpha\lambda(\theta)$.
\end{remark}
In the following, we present an example where the equality is attained in the $\alpha$-CRLB.
\begin{example}\label{exm:optimal_estimator_example_sample_mean}
     Let $\mathrm{X}=(X_1,\ldots, X_n)$ be distributed according to a multivariate t-distribution $t_{n,\nu}(\mu\mathbf{1}_n,\mathbb{I}_n),$ $\mu\in\mathbb{R}$, $\nu>2$ whose density is given by
    \begin{equation}
        \label{eqn:mult_t_dist_mu}
       p_{\mu}(\mathrm{x})= {\displaystyle {N_{n,\nu}}\left[1+\frac{1}{\nu}(\mathrm{x}-\mu\mathbf{1}_n)^{\top}(\mathrm{x}-\mu\mathbf{1}_n)\right]^{-(\nu +n)/2}}, \;\mathrm{x}=[x_1,\ldots,x_n]^{\top}\in \mathbb{R}^n,
    \end{equation} where 
$$N_{n,\nu}=\frac {\Gamma \left[(\nu +n)/2\right]}{\Gamma (\nu /2)\nu ^{n/2} \pi^{n/2}}.$$
Note that
\begin{equation}\label{eqn:marginal_student t_mu}
     p_{\mu}(x_i)
=
\frac{\Gamma[(\nu+1)/2]}
{\Gamma\!\left(\nu/2\right)\sqrt{\nu \pi}}
\left\{1+\frac{(x_i-\mu)^2}{\nu}\right\}^{-(\nu+1)/2}\; \text{ for  } i=1,\cdots,n.
\end{equation}
\eqref{eqn:mult_t_dist_mu} can be expressed as follows:
\begin{equation}\label{eqn:alpha_mult_t_mu}
    p_{\mu}(\mathrm{x})=\left[N_{n,\nu}^{-2/(\nu+n)}\sum_{i=1}^nx_i^2+N_{n,\nu}^{-2/(\nu+n)}\left(1+\frac{n\mu^2}{\nu}\right)-N_{n,\nu}^{-2/(\nu+n)}\frac{2n}{\nu}\mu\overline{x}\right]^{-(\nu+n)/2}.
\end{equation} 
Let $\hat{r}(\mathrm{X})=\overline{X}$, where $\overline{X}=\sum_{i=1}^n X_i/n$. Observe that $\mathbb{E}_{\mu}[\overline{X}]=\mu$. \eqref{eqn:alpha_mult_t_mu} is of the form \eqref{eqn:join_alpha_family} with $\theta=\mu$, $\alpha=-2/(\nu+n)$, $N(\mathrm{x})=N_{n,\nu}^{-2/(\nu+n)}\sum_{i=1}^nx_i^2$, $F(\theta)=N_{n,\nu}^{-2/(\nu+n)}\left(1+\frac{n\mu^2}{\nu}\right)$, $\omega(\theta)=-N_{n,\nu}^{-2/(\nu+n)}\frac{2n}{\nu}\mu$, and $\hat{r}(\mathrm{x})=\overline{x}$. Hence, by remark~\ref{remark:eq_alpha_crlb_family}, $\hat{r}(\mathrm{X})$ attains the $\alpha$-CRLB with $\alpha=-2/(\nu+n)$ for the parameter $\mu$.
\end{example}
\begin{remark}
Notice that, in Example~\ref{exm:optimal_estimator_example_sample_mean}, $X_i$'s are identically distributed but not independent. When $X_i$'s are i.i.d. following \eqref{eqn:marginal_student t_mu}, their joint distribution will not have a $B^{(\alpha)}$ form as in \eqref{eqn:join_alpha_family}. As a result, unlike the exponential family case, we will not be able to identify an estimator from its joint distribution that attains the $\alpha$-CRLB for its expected value. In the following, we first illustrate how an estimator attaining the usual CRLB can be identified from the joint distribution of an exponential family, and then show that this characterization generally fails for the $\mathbb{B}^{(\alpha)}$-family and in general for the $\mathbb{B}^{(\varphi)}$-family.

 Let $X_1,\ldots, X_n$ be an i.i.d. sample from an exponential family. Then the joint distribution is
  \begin{align*}
p_{\theta}(\mathrm{x})
&=\prod_{i=1}^{n} p_{\theta}(x_i) \\
&=\prod_{i=1}^{n}
\exp\!\left(
w(x_i)+L(\theta)+\eta(\theta)^{\top}g(x_i)
\right) \\
&=\exp\!\left(
\sum_{i=1}^{n}
\bigl[w(x_i)+L(\theta)+\eta(\theta)^{\top}g(x_i)\bigr]
\right) \\
&=\exp\!\left(
\sum_{i=1}^{n} w(x_i)
+nL(\theta)
+\eta(\theta)^{\top}\sum_{i=1}^{n} g(x_i)
\right),
\end{align*}
which is again an exponential family. Let $\hat{r}(\mathrm{X})=\sum_{i=1}^n g(X_i)$ and $\mathbb{E}_{\theta}[\hat{r}(\mathrm{X})]=r(\theta)$. $\hat{r}(\mathrm{X})$ attains the usual CRLB for the parameter $r(\theta)$ \cite[eq.~(3.13)]{crrao1945}.

However, in the case of $\mathbb{B}^{(\alpha)}$-family,
\begin{align}
p_{\theta}(\mathrm{x})
&=\prod_{i=1}^{n}p_{\theta}(x_i) \nonumber\\
&=\prod_{i=1}^{n}
\left[
w(x_i)+L(\theta)+\eta(\theta)^\top g(x_i)
\right]^{1/\alpha} \nonumber\\
&=
\Bigg[
\prod_{i=1}^{n}
\left\{
w(x_i)+L(\theta)+\eta(\theta)^\top g(x_i)
\right\}
\Bigg]^{1/\alpha} \nonumber\\
&=
\Bigg[
L(\theta)^n
+L(\theta)^{n-1}
\sum_{i=1}^{n}
\left\{
w(x_i)+\eta(\theta)^\top g(x_i)
\right\}
\nonumber\\
&\qquad
+L(\theta)^{n-2}
\sum_{1\leq i<j\leq n}
\left\{
w(x_i)+\eta(\theta)^\top g(x_i)
\right\}
\left\{
w(x_j)+\eta(\theta)^\top g(x_j)
\right\}
\nonumber\\
&\qquad
+\cdots
+\prod_{i=1}^{n}
\left\{
w(x_i)+\eta(\theta)^\top g(x_i)
\right\}
\Bigg]^{1/\alpha}.
\label{eqn:joint_b_alpha_family}
\end{align}
Hence, for $\alpha\neq0$, the joint density generally cannot be represented as a $\mathbb{B}^{(\alpha)}$-family involving only the single additive statistic $\sum_{i=1}^n g(X_i)$.

This explains the difference between the two settings. For an i.i.d. exponential family, the product structure is compatible with the additive structure of the log-density, and the statistic $\sum_{i=1}^n g(X_i)$ provides the corresponding equality case for the usual CRLB. For an i.i.d. $\mathbb{B}^{(\alpha)}$-family, the product structure generates higher-order interaction terms among the observations. Consequently, the equality characterization for the $\alpha$-CRLB based on a single additive statistic cannot, in general, be applied to the joint distribution of an i.i.d. sample.
\end{remark}
\begin{section}{Attainment of the $\varphi$-CRLB}\label{sec:attainment}
  We now see the significance of Theorem~\ref{thm:discrete_phi_crlb} (Theorem~\ref{thm:Phi_CRLB proof by cauchy-schwarz}). 
\begin{theorem}
\label{thm:Bregman_assym_variance}
    Let $\mathrm{X}=(X_1, \ldots, X_n)$ be distributed according to $p_{\theta}$. Suppose we are interested in estimating $r(\theta)=[r_1(\theta),\ldots,r_k(\theta)]^{\top}$, a $k$-dimensional differentiable function of $\theta$. Let $\hat{r}(\mathrm{X})=[\hat{r}_1(\mathrm{X}),\ldots,\hat{r}_k(\mathrm{X})]^{\top}$ be an unbiased estimator of $r(\theta)$ with respect to $p_{\theta}$, that is, $\mathbb{E}_{\theta}[\hat{r}_i]=r_i(\theta)$ for $i=1,\ldots,k$. Suppose $\hat{r}(\mathrm{x})-r(\theta)$ is square-integrable with respect to $p_{\theta}$ and regularity conditions $\text{I}-\text{IV}$ are satisfied. If $\hat{r}(\mathrm{X})$ attains the $\varphi$-CRLB \eqref{eqn:phi_crlb}, then
    \begin{align}
\label{eqn:Bregman_assym_variance}
   V_{\theta}[\hat{r}(\mathrm{X})] = J(\theta)[I^{(\varphi)}_{n}(\theta)]^{-1}J(\theta)^{\top},
\end{align}
where $[I^{(\varphi)}_{n}(\theta)]^{-1}$ is the the asymptotic covariance matrix of the MBDE   evaluated at $(X_1,\ldots,X_n)$ \cite{Roy2019DensityPower}, where
\begin{equation}
[I_n^{(\varphi)}(\theta)]^{-1}
=
[G_n^{(\varphi)}(\theta)]^{-1}
K_n^{(\varphi)}(\theta)
[G_n^{(\varphi)}(\theta)]^{-1},
\end{equation}with
\begin{equation*}
    G_{n}^{(\varphi)}(\theta)
\coloneqq
\mathbb{E}_{\theta}\!\left[
\frac{\partial}{\partial\theta}
\varphi'\!\big(p_{\theta}(\mathrm{X})\big)
\,u_{\theta}(\mathrm{X})^\top
\right], \quad K_{n}^{(\varphi)}(\theta)
\coloneqq
V_{\theta}\!\left[
p_{\theta}(\mathrm{X})
\varphi''\!\big(p_{\theta}(\mathrm{X})\big)
u_{\theta}(\mathrm{X})
\right].
\end{equation*}
\end{theorem}
\begin{proof}
    Assume that equality holds for $\hat{r}(\mathrm{X})$ in $\varphi$-CRLB. From  Theorem~\ref{thm:Phi_CRLB proof by cauchy-schwarz}, we can write
\begin{equation}
\label{reverse equality condition}
   \hat{r}(\mathrm{x})-\mathbb{E}_{\varphi,\theta}[\hat{r}(\mathrm{X})]=J(\theta) [G_n^{(\varphi)}(\theta)]^{-1} p_{\theta}(\mathrm{x})\varphi''\big(p_{\theta}(\mathrm{x})\big)u_{\theta}(\mathrm{x}).
\end{equation}
Taking covariance on both sides of \eqref{reverse equality condition} with respect to $p_{\theta}$, we get
     \begin{align*}
       V_{\theta}[\hat{r}(\mathrm{X})]
        &=J(\theta)[G_n^{(\varphi)}(\theta)]^{-1} V_{\theta}\left[ p_{\theta}(\mathrm{X})\varphi''\big(p_{\theta}(\mathrm{X})\big)u_{\theta}(\mathrm{X})\right]\left[(G_n^{(\varphi)}(\theta))^{-1}\right]^{\top}J(\theta)^\top\\
        &=J(\theta)[G_{n}^{(\varphi)}(\theta)]^{-1} K_{n}^{(\varphi)}(\theta)[G_{n}^{(\varphi)}(\theta)]^{-1}J(\theta)^{\top},
    \end{align*}
    since $G_n^{(\varphi)}(\theta)$ is symmetric.
    Thus,
    \begin{equation*}
        V_{\theta}[\hat{r}(\mathrm{X})] =J(\theta)[I_{n}^{(\varphi)}(\theta)]^{-1}J(\theta)^{\top}.\nonumber
    \end{equation*}
\end{proof}
\begin{remark}
    For $\varphi(t)=t\log t -t+1$, $I_{n}^{(\varphi)}(\theta)$ reduces to the Fisher information matrix $G_n(\theta)$.
\end{remark}
\begin{remark} The above result is analogous to the fact that the covariance matrix of the estimator that attains the usual Cram\'er-Rao inequality for the unknown parameter $\theta$ is equal to the asymptotic covariance matrix of the MLE evaluated for the joint observation vector $(X_1,\ldots,X_n)$.
\end{remark}
\begin{remark}
     With $\varphi(t)=\varphi_{\alpha}(t)$ as in \eqref{eqn:varphi_function} in Theorem~\ref{thm:Bregman_assym_variance}, we get the $\alpha$-version. If equality holds for $\hat{r}(\mathrm{X})$ in the $\alpha$-CRLB \eqref{eqn:alpha_CRLB_ext}, then
\begin{align}
\label{eqn:Basu_assym_variance}
   V_{\theta}[\hat{r}(\mathrm{X})] =J(\theta) [{I^{(\alpha)}_{n}}(\theta)]^{-1}J(\theta)^{\top},
\end{align}where $[{I^{(\alpha)}_{n}}(\theta)]^{-1}$ is the asymptotic covariance matrix of the BHHJ estimator evaluated for $(X_1,\ldots,X_n)$,\cite{BasuHarris1998}, \cite{basu2011statistical},
where 
$$
[{I^{(\alpha)}_{n}}(\theta)]^{-1} = [G_n^{(\alpha)}(\theta)]^{-1}K_n^{(\alpha)}(\theta) [G_n^{(\alpha)}(\theta)]^{-1},
$$
 with
\begin{align*}
    &G_n^{(\alpha)}(\theta)\coloneqq \mathbb{E}_{\theta}\Big[p_{\theta}(\mathrm{X})^{\alpha}u_{\theta}(\mathrm{X}) u_{\theta}(\mathrm{X})^\top\Big],\quad K_n^{(\alpha)}(\theta)\coloneqq V_{\theta}\Big[p_{\theta}(\mathrm{X})^{\alpha}u_{\theta}(\mathrm{X})\Big].
\end{align*}
\end{remark}
The following examples illustrate the $\alpha$-case.
\begin{example}
   In Example~\ref{exm:optimal_estimator_example_sample_mean}, we observed that $\hat{r}(\mathrm{X})=\overline{X}$ attains the $\alpha$-CRLB for $\theta=\mu$. Now we will show that \eqref{eqn:Basu_assym_variance} holds. We can write $p_{\mu}(\mathrm{x})$ as:
   \begin{equation*}
       p_{\theta}(\mathrm{x})=\left[N(\mathrm{x})+F(\theta)+\omega(\theta)\hat{r}(\mathrm{x})\right]^{1/\alpha},
   \end{equation*}where $\alpha$, $N$, $F$, $\omega$ are as in the Example~\ref{exm:optimal_estimator_example_sample_mean}.
We calculate the following:
\begin{align*}
    V_{\theta}[\hat{r}(\mathrm{X})]=V_{\mu}\Big[\sum_{i=1}^nX_i/n\Big]=\frac{1}{n^2}\Big[\sum_{i=1}^nV_{\mu}(X_i)+2\sum_{1\leq i<j\leq n}\operatorname{Cov}_{\mu}(X_i,X_j)\Big]\overset{(a)}{=}\frac{\nu}{n(\nu-2)},
\end{align*}(a) follows from the fact that $V_{\mu}(\mathrm{X})=\frac{\nu}{\nu-2}\mathbb{I}_n$, which implies that $\operatorname{Cov}_{\mu}(X_i,X_j)=0$ for $i\neq j$.
\begin{align*}
    \mathbb{E}_{\theta}&\left[p_{\theta}(\mathrm{X})^{2\alpha}\big\{u_{\theta}(\mathrm{X})\big\}^2\right]=\int p_{\theta}(\mathrm{x})^{2\alpha+1}\left\{u_{\theta}(\mathrm{x})\right\}^2d\mathrm{x}=\int p_{\theta}(\mathrm{x})^{2\alpha-1}\left\{\frac{\partial}{\partial \theta}p_{\theta}(\mathrm{x})\right\}^2d\mathrm{x}\\
    &=\int p_{\theta}(\mathrm{x})^{2\alpha-1}\left\{\frac{1}{\alpha}[N(\mathrm{x})+F(\theta)+\omega(\theta)\hat{r}(\mathrm{x})]^{\frac{1}{\alpha}-1}[F'(\theta)+\omega'(\theta)\hat{r}(\mathrm{x})]\right\}^2d\mathrm{x}\\
    &=\frac{1}{\alpha^2}\int p_{\theta}(\mathrm{x})^{2\alpha-1}\left\{\frac{p_{\theta}(\mathrm{x})}{p_{\theta}(\mathrm{x})^{\alpha}}[F'(\theta)+\omega'(\theta)\hat{r}(\mathrm{x})]\right\}^2d\mathrm{x}\\
    &=\frac{1}{\alpha^2}\int p_{\theta}(\mathrm{x})[F'(\theta)+\omega'(\theta)\hat{r}(\mathrm{x})]^2d\mathrm{x},
\end{align*} and 
\begin{align*}
      \mathbb{E}_{\theta}\big[p_{\theta}(\mathrm{X})^{\alpha}u_{\theta}(\mathrm{X})\big]&=\int p_{\theta}(\mathrm{x})^{1+\alpha}u_{\theta}(\mathrm{x}) d\mathrm{x}=\int p_{\theta}(\mathrm{x})^{\alpha}\frac{\partial}{\partial \theta}p_{\theta}(\mathrm{x})d\mathrm{x}\\
    &=\int p_{\theta}(\mathrm{x})^{\alpha}\frac{1}{\alpha}[N(\mathrm{x})+F(\theta)+\omega(\theta)\hat{r}(\mathrm{x})]^{\frac{1}{\alpha}-1}[F'(\theta)+\omega'(\theta)\hat{r}(\mathrm{x})]d\mathrm{x}\\
    &=\frac{1}{\alpha}\int p_{\theta}(\mathrm{x})^{\alpha}\frac{p_{\theta}(\mathrm{x})}{p_{\theta}(\mathrm{x})^{\alpha}}[F'(\theta)+\omega'(\theta)\hat{r}(\mathrm{x})]d\mathrm{x}\\
    &=\frac{1}{\alpha}\int p_{\theta}(\mathrm{x})[F'(\theta)+\omega'(\theta)\hat{r}(\mathrm{x})]d\mathrm{x}.
\end{align*}
Hence \begin{align*}
    V_{\theta}\big[p_{\theta}(\mathrm{X})^{\alpha}u_{\theta}(\mathrm{X})\big]&=\frac{1}{\alpha^2}\int p_{\theta}(\mathrm{x})[F'(\theta)+\omega'(\theta)\hat{r}(\mathrm{x})]^2d\mathrm{x}-\left\{\frac{1}{\alpha}\int p_{\theta}(\mathrm{x})[F'(\theta)+\omega'(\theta)\hat{r}(\mathrm{x})]d\mathrm{x}\right\}^2\\
    &=\frac{1}{\alpha^2}V_{\theta}[F'(\theta)+\omega'(\theta)\hat{r}(\mathrm{x})]=\frac{1}{\alpha^2}\omega'(\theta)^2 V_{\theta}[\hat{r}(\mathrm{X})].
\end{align*} Now we calculate
\begin{align*}
    \mathbb{E}_{\theta}&\left[p_{\theta}(\mathrm{X})^{\alpha}\left\{u_{\theta}(\mathrm{X})\right\}^2\right]=\int p_{\theta}(\mathrm{x})^{1+\alpha}\left\{\frac{1}{p_{\theta}(\mathrm{x})}\frac{\partial }{\partial \theta}p_{\theta}(\mathrm{x})\right\}u_{\theta}(\mathrm{x})d\mathrm{x}\\
    &=\int p_{\theta}(\mathrm{x})^{\alpha}\frac{1}{\alpha}[N(\mathrm{x})+F(\theta)+\omega(\theta)\hat{r}(\mathrm{x})]^{\frac{1}{\alpha}-1}[F'(\theta)+\omega'(\theta)\hat{r}(\mathrm{x})]u_{\theta}(\mathrm{x})d\mathrm{x}\\
    &=\frac{1}{\alpha}\int p_{\theta}(\mathrm{x})^{\alpha}\frac{p_{\theta}(\mathrm{x})}{p_{\theta}(\mathrm{x})^{\alpha}}[F'(\theta)+\omega'(\theta)\hat{r}(\mathrm{x})]u_{\theta}(\mathrm{x})d\mathrm{x}\\
    &=\frac{1}{\alpha}\int F'(\theta)u_{\theta}(\mathrm{x})p_{\theta}(\mathrm{x})d\mathrm{x}+\frac{1}{\alpha}\int \omega'(\theta)\hat{r}(\mathrm{x})u_{\theta}(\mathrm{x})p_{\theta}(\mathrm{x})d\mathrm{x}\\
    &=\frac{1}{\alpha}F'(\theta)\int u_{\theta}(\mathrm{x})p_{\theta}(\mathrm{x})d\mathrm{x}+\frac{1}{\alpha}\omega'(\theta)\int \hat{r}(\mathrm{x})\frac{\partial}{\partial \theta}p_{\theta}(\mathrm{x})d\mathrm{x}\\
    &=\frac{1}{\alpha}\omega'(\theta)\frac{\partial \mathbb{E}_{\theta}[\hat{r}(\mathrm{X})]}{\partial \theta}.
\end{align*}
Thus we get 
\begin{align*}
    \frac{  V_{\theta}\big[p_{\theta}(\mathrm{X})^{\alpha}u_{\theta}(\mathrm{X})\big]}{\left(  \mathbb{E}_{\theta}\left[p_{\theta}(\mathrm{X})^{\alpha}\big\{u_{\theta}(\mathrm{X})\big\}^2\right]\right)^{2}}=\frac{\omega'(\theta)^2 V_{\theta}[\hat{r}]}{\Big(\omega'(\theta)\frac{\partial \mathbb{E}_{\theta}[\hat{r}(\mathrm{X})]}{\partial \theta}\Big)^2}=\frac{V_{\theta}[\hat{r}]}{\Big(\frac{\partial \mathbb{E}_{\theta}[\hat{r}(\mathrm{X})]}{\partial \theta}\Big)^2}.
\end{align*}
\end{example}
 \begin{example}
     In Example~\ref{exm:equality}, we saw that $\hat{r}(\mathrm{X})=3X$ attains the $\alpha$-CRLB for $\alpha=1$ and $n=1$. Here, we will verify the result \eqref{eqn:Basu_assym_variance} for this estimator. We calculate the following:
     \begin{equation}
         \mathbb{E}_{\theta}[(3X)^2]=(3\theta^2)/2.\nonumber
     \end{equation} 
     So, we get
     \begin{equation}
       V_{\theta}[3X]=\theta^2/2.  \nonumber
     \end{equation}
     Also
     \begin{align*}
         K_1^{(\alpha)}(\theta)&=V_{\theta}\left[\frac{\partial}{\partial \theta}p_{\theta}(X)\right]=V_{\theta}\left[\frac{4X}{\theta^3}-\frac{2}{\theta^2}\right]\\
         &=\frac{16}{9\theta^6}V_{\theta}[3X]=8/(9\theta^4).
     \end{align*}
     So, $[I_1^{(\alpha)}(\theta)]^{-1}=K_1^{(\alpha)}(\theta)/[G_1^{(\alpha)}(\theta)]^2=\frac{8}{9\theta^4}/\big(\frac{4}{3\theta^3}\big)^2=\theta^2/2.$\\
     Hence $V_{\theta}[3X]=[I_1^{(\alpha)}(\theta)]^{-1}$.
 \end{example}
 
 \begin{remark}\label{remark:loss_efficiency}
  Observe that by Matrix Cauchy-Schwarz inequality \begin{equation*}
    V_{\theta}[u_{\theta}(\mathrm{X})]\succeq \operatorname{Cov}_{\theta}\left[\frac{\partial}{\partial\theta}\varphi'(p_{\theta}(\mathrm{X})),u_{\theta}(\mathrm{X})\right]\left(V_{\theta}\left[\frac{\partial}{\partial\theta}\varphi'(p_{\theta}(\mathrm{X}))\right]\right)^{-1}\operatorname{Cov}_{\theta}\left[\frac{\partial}{\partial\theta}\varphi'(p_{\theta}(\mathrm{X})),u_{\theta}(\mathrm{X})\right].
  \end{equation*}This implies 
  \begin{equation*}
      [I_{n}^{(\varphi)}(\theta)]^{-1}\succeq [G_{n}(\theta)]^{-1}.
  \end{equation*}
Thus, the estimator attaining the $\varphi$-CRLB is generally less efficient than the one attaining the usual CRLB.
\end{remark}
We now present the following theorem, which extends the result on the attainability of the usual CRLB by the sufficient statistics of the exponential family.
  \begin{theorem}\label{thm:sufficiency}
      Suppose $p_{\theta}$ forms a $\mathbb{B}^{(\varphi)}$-family as in \eqref{eqn:vector_phi_phi family_equality} with $C(\mathrm{x})=C\in\mathbb{R}$ for all $\mathrm{x}\in\mathcal{X}$. Then \(\hat{r}(\mathrm{X})\) is a sufficient statistic for \(\theta\). Moreover, if $r(\theta)=\mathbb{E}_{\theta}[\hat{r}(\mathrm{X})],$ then \(\hat{r}(\mathrm{X})\) attains the \(\varphi\)-CRLB for the parameter \(r(\theta)\) and hence has variance equal to \eqref{eqn:Bregman_assym_variance}. 
  \end{theorem}
  \begin{proof}
      We can express   
      \begin{equation}\label{eqn:suff_phi_family}
          p_{\theta}(\mathrm{x})=g(\hat{r}(\mathrm{x});\theta)h(\mathrm{x}),
      \end{equation} where $g(\hat{r}(\mathrm{x});\theta)=(\varphi')^{-1}\!\left[C+Z(\theta)+\lambda(\theta)^\top\hat{r}(\mathrm{x})\right]$ and $h(\mathrm{x})=1$. Therefore, by Factorization Theorem \cite[Th.~6.2.6]{casella2024statistical}, $\hat{r}(\mathrm{X})$ is a sufficient statistic for $\theta$. By Theorem~\ref{prop:phi equality family}, $\hat{r}(\mathrm{X})$ attains the $\varphi$-CRLB for the parameter $r(\theta)$. The last conclusion follows from Theorem \ref{thm:Bregman_assym_variance}.
  \end{proof}
  \begin{remark}\label{remark:exponential_sufficient_estimator}
  For $\varphi(t)=t\log t-t+1$, $p_{\theta}$ in Theorem \ref{thm:sufficiency} forms the following exponential family \begin{equation}\label{eqn:std_exponential_family}
          p_{\theta}(\mathrm{x})=\operatorname{exp}[C+Z(\theta)+\lambda(\theta)^{\top}\hat{r}(\mathrm{x})]\qquad \text{for }\mathrm{x}\in \mathcal{X}.
      \end{equation} In this case, $\hat{r}(\mathrm{X})$ is a sufficient statistic for $\theta$ and it attains the usual CRLB for its expected value \cite[eq.~(3.13)]{crrao1945}. 
     Thus, Theorem~\ref{thm:sufficiency} extends the connection between sufficiency and the usual CRLB attainment from the exponential family to the more general $\mathbb{B}^{(\varphi)}$-family.
      \end{remark}
\begin{remark}\label{remark:expo_sufff}
   With $\varphi(t)=\varphi_{\alpha}(t)$ as in \eqref{eqn:varphi_function} in Theorem~\ref{thm:sufficiency}, we get,
    \begin{equation}\label{eqn:std_B_alpha_family}
        p_{\theta}(\mathrm{x})=[f(\theta)+\omega(\theta)^{\top}\hat{r}(\mathrm{x})]^{1/\alpha}\quad \text{for }\mathrm{x}\in\mathcal{X},
    \end{equation}where $f(\theta)=1+\alpha C+\alpha Z(\theta)$, $\omega(\theta)=\alpha\lambda(\theta)$ and
$\hat{r}(\mathrm{X})$ is a sufficient statistic for $\theta$ and attains the $\alpha$-CRLB.
\end{remark}  
\begin{example}\label{exm:optimal_estimator_example_sample_variance}
 Let $\mathrm{X}=(X_1,\ldots, X_n)$ be a sample from multivariate Student t-distribution as in \eqref{eqn:mult_t_dist}. It forms a $B^{(\alpha)}$- family; see \eqref{Student_t_alpha_family} and $T(\mathrm{X})$ is sufficient for $\theta=[\mu,\sigma^2]^{\top}$. Consider $\hat{r}(\mathrm{X})=T(\mathrm{X})$. Note that $\mathbb{E}_{\mu,\sigma^2}[T(\mathrm{X})]=\left[\mu,\mu^2+\frac{\nu}{(\nu-2)}\sigma^2\right]^{\top}$. By Remark~\ref{remark:eq_alpha_crlb_family}, $T(\mathrm{X})$ attains the $\alpha$-CRLB for the parameter $\left[\mu,\mu^2+\frac{\nu}{(\nu-2)}\sigma^2\right]^{\top}$ with $\alpha=-\frac{2}{\nu+n}$.
\end{example}
\begin{remark}
   For exponential family \eqref{eqn:std_exponential_family}, $\hat{r}(\mathrm{X})$ is sufficient for $\theta$, and its covariance matrix attains the usual CRLB. However, for the $\mathbb{B}^{(\varphi)}$-family (other than $\varphi(t)=t\log t-t+1$), covariance matrix of the sufficient statistic $\hat{r}(\mathrm{X})$ does not attain the usual CRLB (see Remark~\ref{remark:loss_efficiency}); rather, it attains the $\varphi$-CRLB.
\end{remark}

\end{section}
\section{Summary and Concluding Remarks}\label{sec:concluding remarks}

In this paper, we derived a generalized Cram\'er-Rao inequality from Bregman divergence. The derivation is based on the information-geometric approach of Amari and Nagaoka, together with the Riemannian metric induced by a Bregman divergence through Eguchi's theory. We refer to the resulting bound as the $\varphi$-CRLB. The $\alpha$-CRLB is obtained as a special case by considering the BHHJ divergence, which belongs to the class of Bregman divergences. We have also provided a direct proof of the $\varphi$-CRLB based on the Cauchy--Schwarz inequality. This latter approach enables us to extend the bound from finite-support models to models with countably infinite or continuous support. The $\varphi$-CRLB is closely related to the asymptotic covariance matrix of the MBDE.
We showed that the estimator that attains the $\varphi$-CRLB has covariance matrix equal to the asymptotic covariance matrix of MBDE evaluated for the joint distribution of $(X_1,\dots,X_n)$(Theorem \ref{thm:Bregman_assym_variance}). We also showed that the sufficient statistics of the $\mathbb{B}^{(\varphi)}$-family attain the $\varphi$-CRLB and as a consequence has covariance matrix equal to the asymptotic covariance matrix of MBDE  evaluated for the joint distribution of $(X_1,\dots,X_n)$ (Theorem \ref{thm:sufficiency}).


 The equality condition for the $\varphi$-CRLB (Theorem \ref{prop:phi equality family}) also reveals an interesting distinction between general $B^{(\varphi)}$ families and i.i.d. models. We showed that equality holds when the joint distribution of the observations belongs to a $B^{(\varphi)}$-family. On the other hand, if the marginal distribution of each observation belongs to a $B^{(\varphi)}$ family and the observations are assumed to be i.i.d., the resulting joint distribution need not belong to the (same) $B^{(\varphi)}$ family (except for $\varphi(t) = t\log t - t + 1$). Consequently, equality in the $\varphi$-CRLB does not, in general, follow merely from the fact that the marginal distributions belong to a $B^{(\varphi)}$ family. This distinction highlights the role of the structure of the joint distribution in the attainment of the bound. This is in contrast to the case of equality in the usual CRLB and the exponential family.
We also note a connection between the present work and the results of Naudts \cite{naudts2004estimators}, who studied estimators, escort probabilities, and exponential families in the context of statistical physics. The $\varphi$-CRLB considered here may be related to the framework developed by Naudts by taking the distribution $P_{\theta}$ appearing in his formulation to be the $\varphi$-escort distribution associated with $p_{\theta}$ in the present work. This connection provides an alternative perspective on the bound and places it within the broader framework of generalized exponential families and escort distributions. We mention this connection also to clarify that the present result is related to, but formulated from a completely different perspective than, the corresponding result in \cite{naudts2004estimators}. 

\section{Acknowledgements}
Satyajit Dhadumia is supported by a Council of Scientific \& Industrial Research (CSIR) fellowship of the Department of Scientific \& Industrial Research, Ministry of Science and Technology, Government of India. The work was also supported in part by the Science \& Engineering Research Board (SERB) of the Department of Science and Technology, India, MATRICS grant (No. MTR/2022/000922).

\bibliographystyle{plain}
    \bibliography{ref}

\section*{Appendix (Supplementary Material)}
\addcontentsline{toc}{section}{Appendix}  
\section*{A. Details of Remark~\ref{rem:phi_family_phi_escort_relation}}
Let $\phi$ be an increasing function on $[0,\infty)$, strictly positive on $(0,\infty)$. Now define the following
    \begin{equation}\label{eqn:log_function}
    \log_{\phi}(t)\coloneqq\int_{1}^t \frac{1}{\phi(y)}dy,\quad t>0.
\end{equation}Denote $F_{\phi}(t)=\int_{1}^t\log_{\phi}(y)dy$. Now, consider the  strictly convex twice differentiable function $\varphi(t)=F_{\phi}(t)-(1-t)F_{\phi}(0)$, $t\in (0,\infty)$. Observe that 
\begin{align*}
    \varphi'(t)=\log_{\phi}(t)+F_{\phi}(0).
\end{align*}
Hence,
\begin{equation}\label{eqn:gen_phi_family_under_phi}
    (\varphi')^{-1}(t)=\operatorname{exp}_{\phi}[t-F_{\phi}(0)].
\end{equation}\\
Thus, under this $\varphi$, we can write \eqref{eqn:phi_family} as follows
\begin{equation}\label{eqn:phi_escort_and_phi_family}
    p_{\theta}(\mathrm{x})=\operatorname{exp}_{\phi}\left[C(\mathrm{x})+Z(\theta)+\lambda(\theta)^\top T(\mathrm{x})-F_{\phi}(0)\right].
\end{equation}
Now with $C(x)=0$, $Z(\theta)=G(\theta)+F_{\phi}(0)$, $\lambda(\theta)=-\theta^k$, $T(\mathrm{x})=c_k(x)$, \eqref{eqn:gen_phi_family_under_phi} reduces to 
$$p_{\theta}(\mathrm{x})=\operatorname{exp}_{\phi}\left[G(\theta)-\theta^k c_k(x)\right],$$ which is the $\phi$-exponential family.
\section*{B. Proof of Theorem~\ref{thm:var_sq_norm} }
    For every $X\in T_{p}(\mathcal{P})$, we have
\[
X\big(\mathbb{E}[A]\big) = \big\langle(\text{grad}\,\mathbb{E}[A])_p, X\big\rangle_{p}.
\]
We observe that
\begin{align*}
     X\big(\mathbb{E}[A]\big)&=\sum_{x}X^{(m)}(x)A(x).
\end{align*} 
Since $X^{(e)}(\mathrm{x})=X^{(m)}(\mathrm{x})/p(\mathrm{x}) \text{ for } \mathrm{x}\in \mathcal{X}$, we can write
$$X\big(\mathbb{E}[A]\big)=\mathbb{E}_{p}\big[ X^{(e)} A\big].$$
Since $\mathbb{E}_{p}[X^{(e)}]=0$, we can write
\begin{align}
\label{eqn:tgt_vector_to_expectation}
     X\big(\mathbb{E}[A]\big)&=\mathbb{E}_{p}\left\{A X^{(e)}-\frac{\mathbb{E}_{p}[X^{(e)}]}{\sum_{\mathrm{y}}\big[\varphi''(p(\mathrm{y}))\big]^{-1}}\mathbb{E}_{p}\left[\frac{A}{p(\mathrm{X})\varphi''\big (p(\mathrm{X})\big)}\right]\right\}\nonumber\\
                       &= \mathbb{E}_{p}\left\{AX^{(e)}-\frac{X^{(e)}}{\sum_{\mathrm{y}}\big[\varphi''(p(\mathrm{y}))\big]^{-1}}\mathbb{E}_{p}\left[\frac{A}{p(\mathrm{X})\varphi''\big (p(\mathrm{X})\big)}\right]\right\}\nonumber\\
                       &=\mathbb{E}_{p}\left\{X^{(e)}\left(A-\frac{1}{\sum_{\mathrm{y}}\big[\varphi''(p(\mathrm{y}))\big]^{-1}}\mathbb{E}_{p}\left[\frac{A}{p(\mathrm{X})\varphi''\big (p(\mathrm{X})\big)}\right]\right)\right\}\nonumber\\
                       & = \mathbb{E}_{p}\left[X^{(e)}\big(A - \mathbb{E}_{\varphi}[A]\big)\right].
\end{align}
Observe that
 \begin{align}
 \label{eqn:alpha_rep_criterion}
     &\mathbb{E}_{p}\bigg\{\frac{1}{p(\mathrm{X})\varphi''\big (p(\mathrm{X})\big)}\Big(A-\mathbb{E}_{\varphi}[A]\Big)\bigg\}\nonumber
     =\mathbb{E}_{p}\bigg\{\frac{A}{p(\mathrm{X})\varphi''\big (p(\mathrm{X})\big)} - \frac{1}{p(\mathrm{X})\varphi''\big (p(\mathrm{X})\big)}\mathbb{E}_{\varphi}[A]\bigg\}\nonumber\\
     &=\mathbb{E}_{p}\left[\frac{A}{p(\mathrm{X})\varphi''\big (p(\mathrm{X})\big)}\right] -\mathbb{E}_{p}\left[\frac{1}{p(\mathrm{X})\varphi''\big (p(\mathrm{X})\big)}\right]\cdot \mathbb{E}_{\varphi}[A]\nonumber\\
     &=\mathbb{E}_{p}\left[\frac{A}{p(\mathrm{X})\varphi''\big (p(\mathrm{X})\big)}\right]-\sum_{\mathrm{x}}\big[\varphi''(p(\mathrm{X}))\big]^{-1}\cdot\mathbb{E}_{\varphi}[A]\nonumber\\
     &=\mathbb{E}_{p}\left[\frac{A}{p(\mathrm{X})\varphi''\big (p(\mathrm{X})\big)}\right]-\frac{\sum_{\mathrm{x}}\big[\varphi''(p(\mathrm{X}))\big]^{-1}}{\sum_{\mathrm{y}}\big[\varphi''(p(\mathrm{y}))\big]^{-1}}\mathbb{E}_{p}\left[\frac{A}{p(\mathrm{X})\varphi''\big (p(\mathrm{X})\big)}\right]\nonumber\\
     &=\mathbb{E}_{p}\left[\frac{A}{p(\mathrm{X})\varphi''\big (p(\mathrm{X})\big)}\right]-\mathbb{E}_{p}\left[\frac{A}{p(\mathrm{X})\varphi''\big (p(\mathrm{X})\big)}\right]\nonumber\\
     &=0.
  \end{align}   
Thus
$$
A -\mathbb{E}_{\varphi}[A] = Y^{(\varphi)} \in T_{p}^{(\varphi)}(\mathcal{P}),
$$
for some $ Y \in T_{p}(\mathcal{P}) $. Consequently, we can write 
$$
X\big(\mathbb{E}[A]\big) = \mathbb{E}_{p}\big[X^{(e)} Y^{(\varphi)}\big].
$$
Therefore,
\begin{equation}
    \label{eqn: metric expectation}  
\mathbb{E}_{p}\big[X^{(e)} Y^{(\varphi)}\big] = \left\langle(\text{grad}\,\mathbb{E}[A])_p, X\right\rangle_{p}.
\end{equation}
Again, from the definition of the $\varphi$-Fisher information metric \eqref{eqn:phi_riemannian_metric}, we see that 
\begin{equation}
\label{eqn:Ex-Innerproduct}
\mathbb{E}_{p}\big[X^{(e)} Y^{(\varphi)}\big] = \left\langle Y, X\right\rangle_{p}, \quad \text{for any } X, Y \in T_{p}(\mathcal{P}).
\end{equation}
From \eqref{eqn: metric expectation} and \eqref{eqn:Ex-Innerproduct}, we get $Y = (\text{grad}\,\mathbb{E}[A])_p \in T_{p}(\mathcal{P}) $. Hence
\begin{align}
    \|(d\mathbb{E}[A])_{p}\|_{p}^{2}&= \left\langle(\text{grad}\,\mathbb{E}[A])_p, (\text{grad}\,\mathbb{E}[A])_p\right\rangle_{p}\nonumber\\
     &= \left\langle Y, Y\right\rangle_{p}\nonumber \\
     &= \mathbb{E}_{p}\big[Y^{(e)} Y^{(\varphi)}\big]\nonumber\\
     &= \mathbb{E}_{p}\big[Y^{(e)} \big(A -\mathbb{E}_{\varphi}[A]\big)\big].\label{eqn:star1}
\end{align}
We have
\begin{align}
\label{eqn:star0}
  Y^{(e)}(\mathrm{x})&= \frac{Y^{(\varphi)}(\mathrm{x})}{p(\mathrm{x})\varphi''(p(\mathrm{X}))}=  \frac{1}{p(\mathrm{x})\varphi''(p(\mathrm{X}))} \Big(A(\mathrm{x}) -\mathbb{E}_{\varphi}[A]\Big).
\end{align}
Using \eqref{eqn:star0} in \eqref{eqn:star1}, we get 
\begin{align}
\|(d\mathbb{E}[A])_{p}\|_{p}^{2}
&= \mathbb{E}_{p}\left[\frac{1}{p(\mathrm{X})\varphi''(p(\mathrm{X}))} \big(A -\mathbb{E}_{\varphi}[A]\big)^2\right]\nonumber \\
&= \sum_{\mathrm{y}} \big[\varphi''(p(\mathrm{y}))\big]^{-1} \mathbb{E}_{\varphi} \big[\big(A -\mathbb{E}_{\varphi}[A]\big)^2\big]\nonumber \\
& =\sum_{\mathrm{y}} \big[\varphi''(p(\mathrm{y}))\big]^{-1}  V_{\varphi}[A].\nonumber
\end{align}
Thus
\begin{equation}
    V_{\varphi}[A] = \frac{ \|(d\mathbb{E}[A])_{p}\|_{p}^{2}}{\sum_{y} \big[\varphi''\big(p(y)\big)\big]^{-1} }.\nonumber
\end{equation}

\section*{C. Proof of Proposition~\ref{prop:discrete equality condition}}

From Corollary~\ref{corollarry 1}, we get that equality in \eqref{eqn:discrete_phi-CRLB} if and only if $$\sum_{i=1}^k c_i(\hat{r}_i(\mathrm{x})-\mathbb{E}_{\varphi,\theta}[\hat{r}_i(\mathrm{X})])\in T_{p_{\theta}}^{(\varphi)}(S)\text{ for all }[c_i]\in\mathbb{R}^k.$$
That is,
\begin{equation}
    \mathbb{E}_{\theta}\bigg[\frac{1}{p_{\theta}(\mathrm{X})\varphi''(p_{\theta}(\mathrm{X}))}\sum_{i=1}^k c_i(\hat{r}_i(\mathrm{X})-\mathbb{E}_{\varphi,\theta}[\hat{r}_i(\mathrm{X})])\bigg]=0\text{ for all } [c_i]\in \mathbb{R}^k.
\end{equation}
This implies that 
$$c^\top \mathbb{E}_{\theta}\bigg[\frac{1}{p_{\theta}(\mathrm{X})\varphi''(p_{\theta}(\mathrm{X}))}(\hat{r}(\mathrm{X})-\mathbb{E}_{\varphi,\theta}[\hat{r}(\mathrm{X})])\bigg]=0.$$
The above is true for every $c\in \mathbb{R}^k$.\\ 
Hence
\begin{equation*}
     \mathbb{E}_{\theta}\bigg[\frac{1}{p_{\theta}(\mathrm{X})\varphi''(\mathrm{X})}(\hat{r}_i(\mathrm{X})-\mathbb{E}_{\varphi,\theta}[\hat{r}_i(\mathrm{X})])\bigg]=0\text{ for all }i.
\end{equation*}That is,
$$\hat{r}_i(\mathrm{x})-\mathbb{E}_{\varphi,\theta}[\hat{r}_i(\mathrm{X})]\in T_{p_{\theta}}^{(\varphi)}(S) \text{ for all }i.$$
This implies that
\begin{align*}
 \hat{r}(\mathrm{x})-\mathbb{E}_{\varphi,\theta}[\hat{r}(\mathrm{X})]
   =&
   \begin{bmatrix}
       \hat{r}_1(\mathrm{x})-\mathbb{E}_{\varphi,\theta}[\hat{r}_1(\mathrm{X})]\\ 
       \vdots\\
      \hat{r}_k(\mathrm{x})-\mathbb{E}_{\varphi,\theta}[\hat{r}_k(\mathrm{X})]
   \end{bmatrix}\\
   =&\begin{bmatrix}
       \sum_{j=1}^k m_{1j}(\theta)p_{\theta}(\mathrm{x})\varphi''\big(p_{\theta}(\mathrm{x})\big)\frac{\partial}{\partial \theta_j}\log p_{\theta}(\mathrm{x})\\
       \vdots\\
       \sum_{j=1}^k m_{kj}(\theta)p_{\theta}(\mathrm{x})\varphi''\big(p_{\theta}(\mathrm{x})\big)\frac{\partial}{\partial \theta_j}\log p_{\theta}(\mathrm{x})
   \end{bmatrix}\\
    =&\begin{bmatrix}
        m_{11}(\theta)&\hdots &m_{1k}(\theta)\\
        \vdots & \vdots& \vdots \\
        m_{k1}(\theta)&\hdots &m_{kk}(\theta)\\
    \end{bmatrix}
    \begin{bmatrix}
       p_{\theta}(\mathrm{x})\varphi''\big(p_{\theta}(\mathrm{x})\big)\frac{\partial}{\partial \theta_1}\log p_{\theta}(\mathrm{x})\\
        \vdots\\
       p_{\theta}(\mathrm{x})\varphi''\big(p_{\theta}(\mathrm{x})\big)\frac{\partial}{\partial \theta_k}\log p_{\theta}(\mathrm{x})
    \end{bmatrix}\\
    &\;\text{ for some } m_{ij}(\theta), \text{ functions of }\theta,\; \text{where } i,j=1,\dots,k\\
    =&\, m(\theta)p_{\theta}(\mathrm{x})\varphi''\big(p_{\theta}(\mathrm{x})\big)u_{\theta}(\mathrm{x}).
\end{align*} That is, 
\begin{equation}
\label{eqn:star2}
     \hat{r}(\mathrm{x})-\mathbb{E}_{\varphi,\theta}[\hat{r}(\mathrm{X})]=m(\theta)p_{\theta}(\mathrm{x})\varphi''\big(p_{\theta}(\mathrm{x})\big)u_{\theta}(\mathrm{x}),
\end{equation}
for some $k\times k$ matrix $m(\theta)$.\\
We now determine $m(\theta)$. Hence
\begin{equation*}
    \hat{r}(\mathrm{x})u_{\theta}(\mathrm{x})^{\top}-\mathbb{E}_{\varphi,\theta}[\hat{r}(\mathrm{X})]u_{\theta}(\mathrm{x})^{\top}=m(\theta)p_{\theta}(\mathrm{x})\varphi''\big(p_{\theta}(\mathrm{x})\big)u_{\theta}(\mathrm{x})u_{\theta}(\mathrm{x})^{\top}.
\end{equation*}Taking expectation with respect to $p_{\theta}$, we obtain
\begin{equation*}
    \int  \hat{r}(\mathrm{x})u_{\theta}(\mathrm{x})^{\top}p_{\theta}(\mathrm{x})d\mu(\mathrm{x})-\mathbb{E}_{\varphi,\theta}[\hat{r}(\mathrm{X})]\int u_{\theta}(\mathrm{x})^{\top}p_{\theta}(\mathrm{x})d\mu(\mathrm{x})=m(\theta)\mathbb{E}_{\theta}\left[\frac{\partial}{\partial \theta}\varphi'(p_{\theta}(\mathrm{X}))u_{\theta}(\mathrm{X})^{\top}\right].
\end{equation*}
Since $\mathbb{E}_{\theta}[u_{\theta}(\mathrm{X})]=0$, the second term on the left-hand side vanishes. Therefore, 
\begin{equation*}
    \frac{\partial \mathbb{E}_{\theta}[\hat{r}(\mathrm{X})]}{\partial \theta}=m(\theta)G_{n}^{(\varphi)}(\theta).
\end{equation*}Hence
\begin{equation*}
   J(\theta)=m(\theta)G_{n}^{(\varphi)}(\theta),
\end{equation*}
which implies, \begin{equation*}
    m(\theta)= J(\theta)[G_{n}^{(\varphi)}(\theta)]^{-1}.
\end{equation*}
Substituting this in \eqref{eqn:star2}, we obtain the equality condition \eqref{eqn:equality phi CRLB_finite support}.

\section*{D. Details of Example \ref{example 1}}\label{A1} 

 \begin{align*}
    R.H.S.&= \frac{1}{\sum_{\mathrm{y}} p_{\theta}(\mathrm{y})^{1-\alpha}} [G_n^{(\alpha)}(\theta)]^{-1}=\frac{1}{[\theta^{1-\alpha}+(1-\theta)^{1-\alpha}]^n\sum_{\mathrm{y}}p_{\theta}(\mathrm{y})^{1+\alpha}\mathbb{E}^*\left[u_{\theta}(\mathrm{X})^2\right]}\\
     &=\frac{1}{[\theta^{1-\alpha}+(1-\theta)^{1-\alpha}]^n[\theta^{1+\alpha}+(1-\theta)^{1+\alpha}]^n\mathbb{E}^*\left[u_{\theta}(\mathrm{X})^2\right]},
\end{align*}
where the expectation $\mathbb{E}^*$ is with respect to the escort distribution $p^*_{\theta}(\mathrm{x}):=p_{\theta}(\mathrm{x})^{1+\alpha}/{\sum_{\mathrm{y}}p_{\theta}(\mathrm{y})^{1+\alpha} }.$\\ \\
Observe that 
\begin{align*}
    \frac{\partial}{\partial\theta}\log p_{\theta}(x)
    =&\frac{\partial}{\partial \theta}\big[x\log\theta+(1-x)\log (1-\theta)\big]\\
    =&\frac{x}{\theta}-\frac{1-x}{1-\theta}=\frac{x-\theta}{\theta(1-\theta)},
\end{align*}
and
\begin{align*}
\mathbb{E}^*[X_1]&=\sum_{x_1=0,1}x_1\frac{\Big(\theta^{x_1}(1-\theta)^{1-x_1}\Big)^{1+\alpha}}{\sum_{x_1=0,1}\Big(\theta^{x_1}(1-\theta)^{1-x_1}\Big)^{1+\alpha}}\\
&=\frac{\theta^{1+\alpha}}{[\theta^{1+\alpha}+(1-\theta)^{1+\alpha}]}.
\end{align*}
Hence $p^*_{\theta}$ is Bernoulli with the parameter $\theta^{(\alpha)}\coloneqq\frac{\theta^{1+\alpha}}{\theta^{1+\alpha}+(1-\theta)^{1+\alpha}}$.
Note that
\begin{align*}
    \mathbb{E}^*\left[u_{\theta}(\mathrm{X})^2\right]&=\mathbb{E}^*\left[\left(\sum_{i=1}^n\frac{\partial}{\partial \theta}\log p_{\theta}(X_i)\right)^2\right]=\mathbb{E}^*\left[\left(\sum_{i=1}^n \frac{X_i-\theta}{\theta(1-\theta)}\right)^2\right]=\mathbb{E}^*\left[\left(\frac{n(\overline{X}-\theta)}{\theta(1-\theta)}\right)^2\right].
\end{align*}
That is,
\begin{align*}
     \mathbb{E}^*\left[u_{\theta}(\mathrm{X})^2\right] &=\frac{n^2}{\theta^2(1-\theta)^2}\mathbb{E}^*\Big[(\overline{X}-\theta)^2\Big]=\frac{n^2}{\theta^2(1-\theta)^2}\mathbb{E}^*\Big[(\overline{X}-\theta^{(\alpha)}+\theta^{(\alpha)}-\theta)^2\Big]\\
     &=\frac{n^2}{\theta^2(1-\theta)^2}\mathbb{E}^*\Big[(\overline{X}-\theta^{(\alpha)})^2+2(\overline{X}-\theta^{(\alpha)})(\theta^{(\alpha)}-\theta)+(\theta^{(\alpha)}-\theta)^2\Big]\\
     &=\frac{n^2}{\theta^2(1-\theta)^2}\mathbb{E}^*\Big[(\overline{X}-\theta^{(\alpha)})^2+(\theta^{(\alpha)}-\theta)^2\Big]\\
     & = \frac{n^2}{\theta^2(1-\theta)^2}\left\{\mathbb{E}^*[(\overline{X}-\theta^{(\alpha)})^2]+(\theta^{(\alpha)}-\theta)^2\right\}\\
     &=\frac{n^2}{\theta^2(1-\theta)^2}\Big\{V^*[\overline{X}]+(\theta^{(\alpha)}-\theta)^2\Big\},\text{ the variance $V^*$ is with respect to $p_{\theta}^*$}\\
    &=\frac{n^2}{\theta^2(1-\theta)^2}\left\{\frac{\theta^{(\alpha)}(1-\theta^{(\alpha)})}{n}+(\theta^{(\alpha)}-\theta)^2\right\}\\
    &=\frac{n^2}{\theta^2(1-\theta)^2}\left[\frac{\theta^{1+\alpha}(1-\theta)^{1+\alpha}}{n(\theta^{1+\alpha}+(1-\theta)^{1+\alpha})^2}+\Big(\frac{\theta^{1+\alpha}}{\theta^{1+\alpha}+(1-\theta)^{1+\alpha}}-\theta\Big)^2\right]\\
     &=\frac{n^2}{\theta^2(1-\theta)^2}\left[\frac{\theta^{1+\alpha}(1-\theta)^{1+\alpha}}{n(\theta^{1+\alpha}+(1-\theta)^{1+\alpha})^2}+\Big(\frac{\theta^{1+\alpha}-\theta(\theta^{1+\alpha}+(1-\theta)^{1+\alpha})}{\theta^{1+\alpha}+(1-\theta)^{1+\alpha}}\Big)^2\right]\\
      &=\frac{n}{\theta^2(1-\theta)^2}\left[\frac{\theta^{1+\alpha}(1-\theta)^{1+\alpha}+n\theta^{2}(1-\theta)^2(\theta^{\alpha}-(1-\theta)^{\alpha})^2}{(\theta^{1+\alpha}+(1-\theta)^{1+\alpha})^2}\right].\\
\end{align*}
\begin{align*}
    &R.H.S. =\Bigg\{(\theta^{1-\alpha}+(1-\theta)^{1-\alpha})^n(\theta^{1+\alpha}+(1-\theta)^{1+\alpha})^n\\
    &\quad \quad\quad\quad\times\frac{n}{\theta^2(1-\theta)^2}\left[\frac{\theta^{1+\alpha}(1-\theta)^{1+\alpha}+n\theta^{2}(1-\theta)^2(\theta^{\alpha}-(1-\theta)^{\alpha})^2}{(\theta^{1+\alpha}+(1-\theta)^{1+\alpha})^2}\right]\Bigg\}^{-1}\\
    &=\frac{\theta^2(1-\theta)^2(\theta^{1+\alpha}+(1-\theta)^{1+\alpha})^2}{n(\theta^{1-\alpha}+(1-\theta)^{1-\alpha})^n(\theta^{1+\alpha}+(1-\theta)^{1+\alpha})^n[\theta^{1+\alpha}(1-\theta)^{1+\alpha}+n\theta^{2}(1-\theta)^2(\theta^{\alpha}-(1-\theta)^{\alpha})^2]}.\\ \\
   & L.H.S.=V_{\alpha,\theta}[\overline{X}]=V_{\alpha,\theta}\left[\sum_{i=1}^n X_{i}\Big/n\right]=\frac{1}{n}V_{\alpha,\theta}[X_1]\\
    &=\frac{1}{n}\cdot\frac{\theta^{1-\alpha}}{\theta^{1-\alpha}+(1-\theta)^{1-\alpha}}\cdot\frac{(1-\theta)^{1-\alpha}}{\theta^{1-\alpha}+(1-\theta)^{1-\alpha}}=\frac{\theta^{1-\alpha}(1-\theta)^{1-\alpha}}{n(\theta^{1-\alpha}+(1-\theta)^{1-\alpha})^2},
\end{align*}as $p_{\alpha,\theta}$ is Bernoulli with parameter $\theta^{(-\alpha)}\coloneqq\frac{\theta^{1-\alpha}}{\theta^{1-\alpha}(1-\theta)^{1-\alpha}}$.\\
We need to verify if $L.H.S.> R.H.S.$ for $\alpha\neq0$, that is, 
\begin{align*}
   & \frac{\theta^{1-\alpha}(1-\theta)^{1-\alpha}}{n(\theta^{1-\alpha}+(1-\theta)^{1-\alpha})^2}\\
   & > \frac{\theta^2(1-\theta)^2(\theta^{1+\alpha}+(1-\theta)^{1+\alpha})^2}{n(\theta^{1-\alpha}+(1-\theta)^{1-\alpha})^n(\theta^{1+\alpha}+(1-\theta)^{1+\alpha})^n[\theta^{1+\alpha}(1-\theta)^{1+\alpha}+n\theta^{2}(1-\theta)^2(\theta^{\alpha}-(1-\theta)^{\alpha})^2]}.
\end{align*}
That is, 
\begin{align*}
    &\frac{\theta^{1-\alpha}(1-\theta)^{1-\alpha}}{(\theta^{1-\alpha}+(1-\theta)^{1-\alpha})^2}\\
    &> \frac{\theta(1-\theta)(\theta^{1+\alpha}+(1-\theta)^{1+\alpha})^2}{(\theta^{1-\alpha}+(1-\theta)^{1-\alpha})^n(\theta^{1+\alpha}+(1-\theta)^{1+\alpha})^n[\theta^{\alpha}(1-\theta)^{\alpha}+n\theta(1-\theta)(\theta^{\alpha}-(1-\theta)^{\alpha})^2]}.
\end{align*}
That is,
\begin{align*}
    &\frac{\theta^{-\alpha}(1-\theta)^{-\alpha}}{(\theta^{1-\alpha}+(1-\theta)^{1-\alpha})^2}\\
     &>\frac{(\theta^{1+\alpha}+(1-\theta)^{1+\alpha})^2}{(\theta^{1-\alpha}+(1-\theta)^{1-\alpha})^n(\theta^{1+\alpha}+(1-\theta)^{1+\alpha})^n[\theta^{\alpha}(1-\theta)^{\alpha}+n\theta(1-\theta)(\theta^{\alpha}-(1-\theta)^{\alpha})^2]}.
\end{align*}
Observe that 
\begin{align*}
   0< \theta^{\alpha}(1-\theta)^{\alpha}&\leq\theta^{\alpha}(1-\theta)^{\alpha}+n\theta(1-\theta)(\theta^{\alpha}-(1-\theta)^{\alpha})^2.
    \end{align*}
    This implies
    \begin{align*}
 \frac{1}{\theta^{\alpha}(1-\theta)^{\alpha}}&\geq \frac{1}{\theta^{\alpha}(1-\theta)^{\alpha}+n\theta(1-\theta)(\theta^{\alpha}-(1-\theta)^{\alpha})^2}>0.
\end{align*}
That is,
\begin{equation}
    \label{e1}
    \theta^{-\alpha}(1-\theta)^{-\alpha}\geq\frac{1}{\theta^{\alpha}(1-\theta)^{\alpha}+n\theta(1-\theta)(\theta^{\alpha}-(1-\theta)^{\alpha})^2}>0.
\end{equation}
Observe that
\begin{align*}
    &(\theta^{1-\alpha}+(1-\theta)^{1-\alpha})(\theta^{1+\alpha}+(1-\theta)^{1+\alpha})\\
    &=\theta^2+\theta^{1-\alpha}(1-\theta)^{1+\alpha}+(1-\theta)^2+\theta^{1+\alpha}(1-\theta)^{1-\alpha}\\
     &=\theta^2+(1-\theta)^2+\theta(1-\theta)(\theta^{-\alpha}(1-\theta)^{\alpha}+\theta^{\alpha}(1-\theta)^{-\alpha})\\
     &=\theta^2+(1-\theta)^2+2\theta(1-\theta)-2\theta(1-\theta)+\theta(1-\theta)(\theta^{-\alpha}(1-\theta)^{\alpha}+\theta^{\alpha}(1-\theta)^{-\alpha})\\
     &=\theta^2+(1-\theta)^2+2\theta(1-\theta)+\theta(1-\theta)\left[\frac{(1-\theta)^{\alpha}}{\theta^{\alpha}}+\frac{\theta^{\alpha}}{(1-\theta)^{\alpha}}-2\right]\\
      &=\theta^2+(1-\theta)^2+2\theta(1-\theta)+\theta(1-\theta)\left[\frac{\theta^{2\alpha}+(1-\theta)^{2\alpha}-2\theta^{\alpha}(1-\theta)^{\alpha}}{\theta^{\alpha}(1-\theta)^{\alpha}}\right]\\
      &=(\theta+(1-\theta))^2+\theta(1-\theta)\frac{(\theta^{\alpha}-(1-\theta)^{\alpha})^2}{\theta^{\alpha}(1-\theta)^{\alpha}}\\
      &=1+\theta(1-\theta)\frac{(\theta^{\alpha}-(1-\theta)^{\alpha})^2}{\theta^{\alpha}(1-\theta)^{\alpha}}>1.
\end{align*}
Hence
\begin{align*}
     \Big[\Big(\theta^{1-\alpha}+(1-\theta)^{1-\alpha}\Big)&\Big(\theta^{1+\alpha}+(1-\theta)^{1+\alpha}\Big)\Big]^n\\ &>\Big[\Big(\theta^{1-\alpha}+(1-\theta)^{1-\alpha}\Big)\Big(\theta^{1+\alpha}+(1-\theta)^{1+\alpha}\Big)\Big]^2
     >0.
\end{align*}
This implies
\begin{equation}
  \label{e2}
  \frac{1}{(\theta^{1-\alpha}+(1-\theta)^{1-\alpha})^2}> \frac{\Big(\theta^{1+\alpha}+(1-\theta)^{1+\alpha}\Big)^2}{\Big(\theta^{1-\alpha}+(1-\theta)^{1-\alpha}\Big)^n\Big(\theta^{1+\alpha}+(1-\theta)^{1+\alpha}\Big)^n}>0.
\end{equation}
From \eqref{e1} and \eqref{e2}, we conclude $L.H.S.> R.H.S.$.\\

\section*{E. Details of Example~\ref{exm:equality}}
For $0<x<\theta$, $u_{\theta}(x)=\frac{1}{(\theta-x)}-\frac{2}{\theta}=\frac{2x-\theta}{\theta(\theta-x)}$. Taking $1-\frac{x}{\theta}=y$, we get
\begin{align*}
    \int_{0}^{\theta}p_{\theta}(x)^{1+\alpha}dx&=\frac{2^{1+\alpha}}{\theta^{\alpha}}\left[\frac{y^{2+\alpha}}{(2+\alpha)}\right]_0^1,\quad\alpha>-2\\
    &=\frac{2^{1+\alpha}}{(2+\alpha)\theta^{\alpha}},\quad\alpha>-2.
\end{align*}
\begin{align*}
    \int_0^{\theta}p_{\theta}(x)^{1+\alpha}u_{\theta}(x)^2dx&=\int_0^1\left(\frac{2}{\theta}y\right)^{1+\alpha}\left(\frac{1-2y}{\theta y}\right)^2\theta dy\\
    &=\frac{2^{1+\alpha}}{\theta^{2+\alpha}}\int_0^1\left(y^{\alpha-1}-4y^{\alpha}+4y^{\alpha+1}\right)dy\\
    &=\frac{2^{1+\alpha}}{\theta^{2+\alpha}}\left[\frac{y^{\alpha}}{\alpha}-4\frac{y^{\alpha+1}}{\alpha+1}+4\frac{y^{\alpha+2}}{\alpha+2}\right]_0^1,\quad \alpha>0\\
    &=\frac{2^{\alpha+1}}{\theta^{\alpha+2}}\left\{\frac{\alpha^2-\alpha+2}{\alpha(\alpha+1)(\alpha+2)}\right\},\quad\alpha>0.
\end{align*}
\begin{align*}
     \int_0^{\theta}p_{\theta}(x)^{1+\alpha}u_{\theta}(x)dx&=\int_0^1\left(\frac{2}{\theta}y\right)^{1+\alpha}\frac{(1-2y)}{\theta y}\theta dy\\
     &=\frac{2^{1+\alpha}}{\theta^{1+\alpha}}\int_0^1(y^{\alpha}-2y^{\alpha+1})dy\\
     &=\frac{2^{1+\alpha}}{\theta^{1+\alpha}}\left[\frac{y^{\alpha+1}}{\alpha+1}-2\frac{y^{\alpha+2}}{\alpha+2}\right]_0^1,\quad\alpha>-1\\
     &=\frac{2^{\alpha+1}}{\theta^{\alpha+1}}\left[\frac{-\alpha}{(\alpha+1)(\alpha+2)}\right],\quad\alpha>-1.
\end{align*}
Hence
\begin{align*}
    G_n^{(\alpha)}(\theta)&=\mathbb{E}_{\theta}\left[p_{\theta}(\mathrm{X})^{\alpha}u_{\theta}(\mathrm{X})^2\right]\\
    &=\int_0^{\theta}\cdots\int_0^{\theta}p_{\theta}(\mathrm{x})^{1+\alpha}\left(\sum_{i=1}^n u_{\theta}(x_i)\right)^2 d\mathrm{x}\\
    &=\int_0^{\theta}\cdots\int_0^{\theta}p_{\theta}(\mathrm{x})^{1+\alpha}\Big[\sum_{i=1}^n u_{\theta}(x_i)^2+\sum_{\substack{1 \le i,j \le n \\ i \neq j}}u_{\theta}(x_i)u_{\theta}(x_j)\Big]d\mathrm{x}\\
    &=\sum_{i=1}^n \int_0^{\theta}\cdots\int_0^{\theta}p_{\theta}(\mathrm{x})^{1+\alpha}u_{\theta}(x_i)^2 d\mathrm{x}+ \sum_{\substack{1 \le i,j \le n \\ i \neq j}}\int_0^{\theta}\cdots\int_0^{\theta}p_{\theta}(\mathrm{x})^{1+\alpha}u_{\theta}(x_i)u_{\theta}(x_j)d\mathrm{x}\\
    &=n \left[\frac{2^{\alpha+1}}{(2+\alpha)\theta^{\alpha}}\right]^{n-1}\frac{2^{\alpha+1}}{\theta^{\alpha+2}}\left\{\frac{\alpha^2-\alpha+2}{\alpha(\alpha+1)(\alpha+2)}\right\}\\
    &\quad+n(n-1)\left[\frac{2^{\alpha+1}}{(2+\alpha)\theta^{\alpha}}\right]^{n-2}\left\{\frac{2^{\alpha+1}}{\theta^{\alpha+1}}\left[\frac{-\alpha}{(\alpha+1)(\alpha+2)}\right]\right\}^{2},\quad \alpha>0\\
    &=n\left\{\frac{2^{\alpha+1}}{\theta^{\alpha}(2+\alpha)}\right\}^{n-2}\frac{2^{2\alpha+2}}{\theta^{2\alpha+2}(\alpha+1)(\alpha+2)^2}\left\{\frac{\alpha^2-\alpha+2}{\alpha}+\frac{(n-1)\alpha^2}{\alpha+1}\right\},\quad \alpha>0\\
    &=\frac{n2^{(\alpha+1)n}(n\alpha^3+\alpha+2)}{\alpha(\alpha+1)^2(\alpha+2)^n\theta^{\alpha n+2}},\quad \alpha>0.
\end{align*} 
Note that as $\alpha\to0$, $G_n^{(\alpha)}(\theta)$ diverges.\\ \\
Observe that
\begin{align*}
p_{\theta}(x)^{1-\alpha}&=\frac{2^{1-\alpha}(\theta-x)^{1-\alpha}}{\theta^{2(1-\alpha)}},\\
    \int_0^{\theta}p_{\theta}(x)^{1-\alpha}dx
    &=\frac{2^{1-\alpha}}{\theta^{-\alpha}}\left[\frac{y^{2-\alpha}}{2-\alpha}\right]_0^1,\;\alpha<2\\
    &=\frac{2^{1-\alpha}}{(2-\alpha)\theta^{-\alpha}},\;\alpha<2.
\end{align*}
Note that for $\alpha\geq 2$, $  \int_0^{\theta}p_{\theta}(x)^{1-\alpha}dx$ diverges, as a result $p_{\alpha,\theta}$ is not a valid pdf.\\
Hence, for $\alpha < 2$, we have
\begin{equation}\label{eqn:alpha_escort_example}
p_{\alpha,\theta}(x)
=
\begin{cases}
\displaystyle \frac{(2-\alpha)(\theta - x)^{1-\alpha}}{\theta^{2-\alpha}} &\quad\text{for  } 0 \leq x < \theta,\\
0 & \quad\text{otherwise}.
\end{cases}
\end{equation}
Hence
\begin{equation*}
    \mathbb{E}_{\alpha,\theta}[X]=\frac{(2-\alpha)}{\theta^{2-\alpha}}\int_0^{\theta}x(\theta-x)^{1-\alpha}dx
\end{equation*} and 
\begin{equation*}
     \mathbb{E}_{\alpha,\theta}[X^2]=\frac{(2-\alpha)}{\theta^{2-\alpha}}\int_0^{\theta}x^2(\theta-x)^{1-\alpha}dx.
\end{equation*}
Observe that, by taking $x=\theta t$, we get 
\begin{align*}
&\int_0^{\theta}x(\theta-x)^{1-\alpha}dx=\int_0^1t\theta^{3-\alpha}(1-t)^{1-\alpha}dt\\
&=\theta^{3-\alpha}\int_0^1t(1-t)^{1-\alpha}dt\\
&=\theta^{3-\alpha}\int_0^1 \{1-(1-t)\}(1-t)^{1-\alpha}dt\\
&=\theta^{3-\alpha}\left[-\int_0^1 (1-t)^{1-\alpha}d(1-t)+\int_0^1 (1-t)^{2-\alpha}d(1-t)\right]\\
&=\theta^{3-\alpha}\left[-\frac{(1-t)^{2-\alpha}}{(2-\alpha)}+\frac{(1-t)^{3-\alpha}}{3-\alpha}\right]_0^1,\quad \alpha<2 \\
&=\theta^{3-\alpha}\left[\frac{1}{2-\alpha}-\frac{1}{3-\alpha}\right]=\frac{\theta^{3-\alpha}}{(2-\alpha)(3-\alpha)},\quad \alpha<2.
\end{align*}
\begin{align*}
    &\int_0^{\theta}x^2(\theta-x)^{1-\alpha}dx=\int_0^1t^2\theta^{4-\alpha}(1-t)^{1-\alpha}dt\\
    &=\theta^{4-\alpha}\int_0^1(1-2(1-t)+(1-t)^2)(1-t)^{1-\alpha}dt\\
    &=\theta^{4-\alpha}\left\{\int_0^1(1-t)^{1-\alpha}dt-2\int_0^1 (1-t)^{2-\alpha}dt+\int_0^1 (1-t)^{3-\alpha}dt\right\}\\
    &=\theta^{4-\alpha}\left\{-\int_0^1(1-t)^{1-\alpha}d(1-t)+2\int_0^1 (1-t)^{2-\alpha}d(1-t)-\int_0^1 (1-t)^{3-\alpha}d(1-t)\right\}\\
    &=\theta^{4-\alpha}\left[-\frac{(1-t)^{2-\alpha}}{2-\alpha}+2\frac{(1-t)^{3-\alpha}}{3-\alpha}-\frac{(1-t)^{4-\alpha}}{4-\alpha}\right]_0^1,\quad \alpha<2\\
    &=\theta^{4-\alpha}\left[\frac{1}{2-\alpha}-\frac{2}{3-\alpha}+\frac{1}{4-\alpha}\right],\quad \alpha<2\\
    &=\theta^{4-\alpha}\frac{2}{(2-\alpha)(3-\alpha)(4-\alpha)},\quad \alpha<2.
\end{align*}
Thus, $\mathbb{E}_{\alpha,\theta}[X]=\frac{\theta}{(3-\alpha)}$ for $\alpha<2$ and $\mathbb{E}_{\alpha,\theta}[X^2]=\frac{2\theta^2}{(3-\alpha)(4-\alpha)}$ for $\alpha<2$.
Hence \begin{align*}
    V_{\alpha,\theta}[3\overline{X}]=\frac{9}{n}V_{\alpha,\theta}[X]=\frac{9}{n}\frac{(2-\alpha)\theta^2}{(3-\alpha)^2(4-\alpha)},\quad \alpha<2.
\end{align*}
Thus, for $0<\alpha<2$, 
L.H.S. of $\alpha$-CRLB 
\[
 V_{\alpha,\theta}[3\overline{X}]=\frac{9}{n}\frac{(2-\alpha)\theta^2}{(3-\alpha)^2(4-\alpha)}\eqqcolon A_n(\theta)
 \]
and R.H.S. of $\alpha$-CRLB 
\[
\frac{1}{\int_0^{\theta}\cdots\int_0^{\theta}p_{\theta}(\mathrm{x})^{1-\alpha}d\mathrm{x}}[G_n^{(\alpha)}(\theta)]^{-1}=\frac{(2-\alpha)^n\alpha(\alpha+1)^2(\alpha+2)\theta^2}{2^{2n}n(n\alpha^3+\alpha+2)}\eqqcolon B_n(\theta).
\]
Now, we will prove $A_n(\theta)\geq B_n(\theta)$. Since $0<\alpha<2$, $0<\frac{2-\alpha}{4}<1$. This implies $0<\frac{(2-\alpha)^n}{4^n}\leq \frac{2-\alpha}{4}<1$ and $\alpha^3+\alpha+2\leq n\alpha^3+\alpha+2$. Observe that
\begin{align*}
   B_n(\theta)= \frac{(2-\alpha)^n\alpha(\alpha+1)^2(\alpha+2)\theta^2}{2^{2n}n(n\alpha^3+\alpha+2)}&\leq \frac{(2-\alpha)\alpha(\alpha+1)^2(\alpha+2)\theta^2}{4n(n\alpha^3+\alpha+2)}\\ 
    &\leq \frac{(2-\alpha)\alpha(\alpha+1)^2(\alpha+2)\theta^2}{4n(\alpha^3+\alpha+2)}\\
    &=\frac{(2-\alpha)\alpha(\alpha+1)(\alpha+2)\theta^2}{4n(\alpha^2-\alpha+2)}.
\end{align*}
We now verify whether
$
A_n(\theta)\geq \frac{(2-\alpha)\alpha(\alpha+1)(\alpha+2)\theta^2}{4n(\alpha^2-\alpha+2)}
$
holds. That is,
\begin{align*}
    \frac{9}{(3-\alpha)^2(4-\alpha)}\geq \frac{\alpha(\alpha+1)(\alpha+2)}{4(\alpha^2-\alpha+2)}.
\end{align*} That is, 
\begin{align*}
   f(\alpha)\coloneqq  36(\alpha^2-\alpha+2)-\alpha(\alpha+1)(\alpha+2)(3-\alpha)^2(4-\alpha)\geq0.
\end{align*}
Note that $f(\alpha)=(\alpha-1)^2(\alpha^4-5\alpha^3-6\alpha^2+36\alpha+72).$
Since $0<\alpha<2$, therefore we have $(\alpha-1)^2\geq 0$, $\alpha^4> 0$, $-5\alpha^3>-5\cdot 2^3=-40$, $-6\alpha^2> -6\cdot 2^2=-24$ and $36\alpha>0$. Thus,
\begin{align*}
   \alpha^4-5\alpha^3-6\alpha^2+36\alpha+72>0-40-24-0+72=8.
\end{align*}
Hence $f(\alpha)\geq0$. 

\section*{F. Details of Proposition~\ref{prop:canonical_phi_family}}
 Since $\lambda$ is continuously differentiable and $R(\theta)=\frac{\partial \lambda}{\partial\theta}$ is invertible, from Inverse function theorem, we get $\left[\frac{\partial \lambda}{\partial\theta}\right]^{-1}=\frac{\partial\theta}{\partial\lambda}$.
\\
From \eqref{canonical_phi_family}, we get,
\begin{equation}
      \frac{\partial}{\partial \lambda}\varphi'(p_{\lambda}(\mathrm{x}))=   \frac{\partial H(\lambda)}{\partial \lambda}+\hat{r}(\mathrm{x}).
\end{equation} This gives
\begin{equation*}
    \mathbb{E}_{\varphi,\lambda}[\hat{r}(\mathrm{X})]=-\frac{\partial H(\lambda)}{\partial\lambda}.
\end{equation*}
Observe that
\begin{align*}
    G_n^{(\varphi)}(\lambda)&=\mathbb{E}_{\lambda}\left[\frac{\partial}{\partial\lambda}\varphi'(p_{\lambda}(\mathrm{X}))\frac{\partial}{\partial\lambda}\log p_{\lambda}(\mathrm{X})^{\top}\right]\\
    &=\int\left\{\frac{\partial H(\lambda)}{\partial \lambda}+\hat{r}(\mathrm{x})\right\}\frac{\partial}{\partial \lambda}p_{\lambda}(\mathrm{x})^{\top}d\mu(\mathrm{x})\\
    &=\frac{\partial}{\partial\lambda}\mathbb{E}_{\lambda}[\hat{r}(\mathrm{X})].
\end{align*}
\begin{align*}
    V_{\varphi,\lambda}[\hat{r}(\mathrm{X})]&=\mathbb{E}_{\varphi,\lambda}[(\hat{r}(\mathrm{X})-\mathbb{E}_{\varphi,\lambda}[\hat{r}(\mathrm{X})])(\hat{r}(\mathrm{X})-\mathbb{E}_{\varphi,\lambda}[\hat{r}(\mathrm{X})])^{\top}]\\
    &=\mathbb{E}_{\varphi,\lambda}\left[\left(\hat{r}(\mathrm{X})+\frac{\partial H(\lambda)}{\partial \lambda}\right)\left(\hat{r}(\mathrm{X})+\frac{\partial H(\lambda)}{\partial \lambda}\right)^{\top}\right]\\
    &=\int \left(\hat{r}(\mathrm{X})+\frac{\partial H(\lambda)}{\partial \lambda}\right)\frac{\partial}{\partial\lambda}\varphi'(p_{\lambda}(\mathrm{x}))^{\top}p_{\varphi,\lambda}(\mathrm{x})d\mu(\mathrm{x})\\
    &=\int \left(\hat{r}(\mathrm{X})+\frac{\partial H(\lambda)}{\partial \lambda}\right)\frac{\partial}{\partial\lambda}p_{\lambda}(\mathrm{x})^{\top}\frac{1}{\int [\varphi''(p_{\lambda}(\mathrm{y}))]^{-1}d\mu(\mathrm{y})}d\mu(\mathrm{x})\\
    &=\frac{1}{\int [\varphi''(p_{\lambda}(\mathrm{y}))]^{-1}d\mu(\mathrm{y})}\frac{\partial}{\partial\lambda}\mathbb{E}_{\lambda}[\hat{r}(\mathrm{X})].
\end{align*}
Now,
\begin{align}
      V_{\varphi,\lambda}[\hat{r}(\mathrm{X})]&=\frac{1}{\int [\varphi''(p_{\lambda}(\mathrm{y}))]^{-1}d\mu(\mathrm{y})}\frac{\partial}{\partial\lambda}\mathbb{E}_{\lambda}[\hat{r}(\mathrm{X})][(G_n^{(\varphi)}(\lambda))]^{-1}[G_n^{(\varphi)}(\lambda)]^{\top}\nonumber\\
      &=\frac{1}{\int [\varphi''(p_{\lambda}(\mathrm{y}))]^{-1}d\mu(\mathrm{y})}\frac{\partial}{\partial\lambda}\mathbb{E}_{\lambda}[\hat{r}(\mathrm{X})][G_n^{(\varphi)}(\lambda)]^{-1}\left(\frac{\partial}{\partial\lambda}\mathbb{E}_{\lambda}[\hat{r}(\mathrm{X})]\right)^{\top}.\label{eqn:m1}
\end{align}
Note that 
\begin{align*}
    G_n^{(\varphi)}(\lambda)&=\mathbb{E}_{\lambda}\left[\frac{\partial}{\partial\lambda}\varphi'(p_{\lambda}(\mathrm{X}))\frac{\partial}{\partial\lambda}\log p_{\lambda}(\mathrm
X)^{\top}\right]\\
&=\mathbb{E}_{\lambda}\left[\frac{\partial\theta}{\partial\lambda}^{\top}\frac{\partial}{\partial\theta}\varphi'(p_{\theta}(\mathrm{X}))\frac{\partial}{\partial\theta}\log p_{\theta}(\mathrm
X)^{\top}\frac{\partial\theta}{\partial\lambda}\right]\\
&=\frac{\partial\theta}{\partial\lambda}^{\top}\mathbb{E}_{\lambda}\left[\frac{\partial}{\partial\theta}\varphi'(p_{\theta}(\mathrm{X}))\frac{\partial}{\partial\theta}\log p_{\theta}(\mathrm
X)^{\top}\right]\frac{\partial\theta}{\partial\lambda}\\
&=\left[\left(\frac{\partial \lambda}{\partial\theta}\right)^{\top}\right]^{-1}\mathbb{E}_{\theta}\left[\frac{\partial}{\partial\theta}\varphi'(p_{\theta}(\mathrm{X}))\frac{\partial}{\partial\theta}\log p_{\theta}(\mathrm
X)^{\top}\right] \left[\frac{\partial \lambda}{\partial\theta}\right]^{-1}.
\end{align*}
From \eqref{eqn:m1}, we get,
\begin{align*}
    V_{\varphi,\lambda}[\hat{r}(\mathrm{X})]=\frac{1}{\int [\varphi''(p_{\theta}(\mathrm{y}))]^{-1}d\mu(\mathrm{y})}\frac{\partial}{\partial\theta}\mathbb{E}_{\theta}[\hat{r}(\mathrm{X})][G_n^{(\varphi)}(\theta)]^{-1}\frac{\partial}{\partial\theta}\mathbb{E}_{\theta}[\hat{r}(\mathrm{X})]^{\top}.
\end{align*}
\end{document}